\documentclass {article}
\usepackage{etex}
\usepackage[utf8]{inputenc}
\usepackage[english]{babel}
\usepackage{amsmath}
\usepackage{amsthm}
\usepackage{amssymb}
\usepackage{sectsty}
\usepackage{titlesec}
\usepackage{color}
\usepackage[color,matrix,arrow]{xy}
\usepackage{amsgen}
\usepackage{amstext}
\usepackage{amsbsy}
\usepackage{amsopn}
\usepackage{amsfonts}
\usepackage{eepic}
\usepackage{graphicx}
\usepackage{epsf}
\usepackage{pstricks}
\xyoption{all}

\usepackage{tikz}
\usepackage{pgf}

\usepackage{hyperref}

\newcommand{\footrecall}[1]{%
}

\def\mapright#1{\smash{\mathop{\longrightarrow}\limits^{#1}}}
\def\tra#1{\smash{\mathop{\mid\kern
-1pt\joinrel\relbar\joinrel\relbar}\limits^{*}_{#1}}}
\def\longtra#1{\smash{\mathop{\mid\kern
-1pt\joinrel\relbar\joinrel\relbar\joinrel\relbar}\limits^{*}_{#1}}}
\def\vlongtra#1{\smash{\mathop{\mid\kern
-1pt\joinrel\relbar\joinrel\relbar\joinrel\relbar\joinrel\relbar}\limits^{*}_{#1}}}
\def\vvlongtra#1{\smash{\mathop{\mid\kern
-1pt\joinrel\relbar\joinrel\relbar\joinrel\relbar\joinrel\relbar\joinrel\relbar}\limits^{*}_{#1}}}
\def\vvvlongtra#1{\smash{\mathop{\mid\kern
-1pt\joinrel\relbar\joinrel\relbar\joinrel\relbar\joinrel\relbar\joinrel\relbar\joinrel\relbar}\limits^{*}_{#1}}}
\def\etra#1{\smash{\mathop{\mid\kern
-1pt\joinrel\relbar\joinrel\relbar}\limits_{#1}}}

\def\vlongrightarrow{\relbar\joinrel\longrightarrow}

\def\vvlongrightarrow{\relbar\joinrel\vlongrightarrow}
\def\vvvlongrightarrow{\relbar\joinrel\vvlongrightarrow}
\def\vvvvlongrightarrow{\relbar\joinrel\vvvlongrightarrow}
\def\vvvvvlongrightarrow{\relbar\joinrel\vvvvlongrightarrow}

\def\vvlongmapright#1{\smash{\mathop{\vvvlongrightarrow}\limits^{#1}}}

\def\vvvvlongmapright#1{\smash{\mathop{\vvvvvlongrightarrow}\limits^{#1}}}

\titleformat*{\section}{\large\bfseries}
\titleformat*{\subsection}{\normalsize \bfseries}
\title{On finitely generated submonoids of virtually free groups}
\date{}
\begin{document}

\newtheorem{theorem}{Theorem}[section]
\newtheorem{lemma}[theorem]{Lemma}
\newtheorem{question}[theorem]{Question}
\newtheorem{proposition}[theorem]{Proposition}
\newtheorem{corollary}[theorem]{Corollary}

\theoremstyle{definition}
\newtheorem{definition}[theorem]{Definition}
\newtheorem{example}[theorem]{Example}

\newcommand{\ophi}{\overline{\varphi}}
\newcommand{\opsi}{\overline{\psi}}
\newcommand{\wt}{\widetilde}

%\newtheorem*{corollary}{Corollary}
%\newtheorem*{definition}{Definition}

%\maketitle
%\vspace{-0.4 cm}

\begin{center}
{\Large On finitely generated submonoids of virtually free groups
}
\par \vspace{0.5 cm}
 Pedro V. Silva\footnote{E-mail address: pvsilva@fc.up.pt}, Alexander Zakharov\footnote{E-mail address: zakhar.sasha@gmail.com}

\par

\end{center}

{\it \small Centre of Mathematics, University of Porto,
R. Campo Alegre 687, 4169-007 Porto, Portugal}
\vspace{0.2 cm}

\begin{abstract}
We prove that it is decidable whether or not a finitely generated submonoid of a virtually free group is graded, introduce a new geometric characterization as quasi-geodesic monoids, and show that their word problem is rational (as a relation). We also solve the isomorphism problem for this class of monoids, generalizing earlier results for submonoids of free monoids. We also prove that the classes of graded monoids, regular monoids and Kleene monoids coincide for submonoids of free groups.
	
	\par 
	\vspace{0.2 cm}
	
	
	\par \vspace{0.2 cm}
\end{abstract}

\vspace{0.2 cm}

\section{Introduction}

A lot is known about subgroups of free groups. And the seminal work of Stallings \cite{Sta} provided the tools to expand that knowledge in multiple directions, particularly when automata got into the picture (see \cite{BS1}, \cite{KM}). This glorious picture fades away when we change from subgroups to submonoids, even in the finitely generated case. Indeed, we cannot use anymore inverse automata and that is a major setback.
The membership problem is still decidable, as a consequence of Benois' Theorem (see \cite[Theorem 6.1, p.314]{Sak}), but e.g. the isomorphism problem remains open. 

On the other hand, submonoids of virtually free groups (i.e. groups admitting a free subgroup of finite index) have been studied in \cite{CRR}.

In this paper, we explore a subclass of finitely generated submonoids of virtually free groups which is somehow in the antipodals of subgroups. If the ambient group is free, we show they can be characterized in alternative equivalent ways, which we now describe:

\begin{itemize}

\item

{\bf Graded submonoids:}

%%%%%%%%%%%%%%%%%%% Changed
those that have a finite system of generators $S$ such that every element of the monoid can be written as a word over $S$ only in finitely many different ways.
%they are finitely generated and each element $w$ admits a finite number of factorizations $w = w_1\ldots w_n$ in the submonoid with $w_i \neq 1$.
%%%%%%%%%%%%%%%%%%%%

\item

{\bf Rational submonoids:}

the structure of the submonoid can be encoded by a finite automaton which recognizes a function associating to every word a normal form.

\item

{\bf Kleene submonoids:}

submonoids where the recognizable subsets coincide with the rational subsets.

\item

{\bf Quasi-geodesic submonoids:}

the continuous paths induced by the submonoid in the Cayley graph of the ambient group are quasi-geodesics in a uniform way.

\end{itemize}
If the ambient group is virtually free, we have 
$$\mbox{graded $\Leftrightarrow$ quasi-geodesic $\Rightarrow$ rational $\Rightarrow$ Kleene}.$$

The concept of graded monoid was introduced by Margolis, Meakin and Sunik in \cite{MMS}. %A submonoid isomorphic to a submonoid of a free monoid is graded, but the converse fails. 
%%%%%%%%%%%%%%  Changed
Graded monoids form an important class of monoids, which includes, in particular, submonoids of free monoids, Artin monoids, many Baumslag-Solitar monoids, as well as monoids satisfying some small cancellation conditions, see \cite{MMS} and references therein  for further details. The class of graded monoids has a remarkably good algorithmic and finite separability theory, in particular, graded monoids are finitely separable, see \cite{MMS}. Together with the closely related notion of upper distortion function, graded monoids were successfully applied in \cite{MMS} in order to solve the membership problem for some submonoids in one-relator groups. 
%%%%%%%%%%%%%%

Rational monoids were introduced by Sakarovitch \cite{Sak2} and they constitute an interesting alternative to automatic monoids: rational monoids are a more restricted class, but they have stronger properties and are arguably more natural from the viewpoint of semigroup theory: there is a rational function from a free monoid into itself which produces a normal form for the monoid.

Kleene's Theorem states that the classes of rational and recognizable languages coincide over a finite alphabet (see \cite[Theorem 2.1, p.87]{Sak}). In general, these two classes do not coincide. For instance, a subgroup of a group is rational (respectively recognizable) if and only if it is finitely generated (respectively finite index) \cite[Exercise III.1.3 and Theorem III.2.7]{Ber}. 

In a geodesic metric space, one can think of a quasi-geodesic as the image of a geodesic by a quasi-isometry. Every path with bounded domain is a quasi-geodesic for adequate parameters, so what defines a quasi-geodesic submonoid is the fact that one can get uniform parameters for all the paths induced by words on the generators of the submonoid.

%Besides establishing the equivalence of the above four properties for finitely generated submonoids of virtually free groups, we prove that these properties are decidable, as well as solve the homomorphism and isomorphism problems for submonoids in this class.

%%%%%%%%%%%%%%%%% Added

The main results of the paper are connected as follows. We use the geometric characterization of graded submonoids as quasi-geodesic submonoids to prove they have a rational word problem. This corresponds to encoding all the relations satisfied by the monoid $M$ by a finite automaton (note that $M$ might be not finitely presented, even if it is a submonoid of a free monoid, see Example \ref{ex4}). Then a map $\varphi: S \rightarrow G$ with $G$ virtually free can be extended to a homomorphism $\overline{\varphi}: M \rightarrow G$ if and only if all the relations in $M$ go to true equalities in $G$ under $\varphi$, and we show this can be decided using the finite automaton encoding the relations and an additional lemma about rational subsets in $G \times G$: one can decide whether such a rational subset lies inside the diagonal. Since a graded submonoid of $G$ has a unique minimal generating set, namely the set of all irreducible elements, we obtain a solution for the isomorphism problem almost immediately.

The paper is organized as follows. In Section 2 we recall some basic definitions from automata theory and formal language theory that are used later, namely, the notions of finite automata, rational and recognizable subsets. In Section 3 we introduce the notion of graded monoid, mention some properties and equivalent characterizations of graded monoids and give some examples of graded and non-graded submonoids of free groups.

In Section 4 we prove that graded submonoids of virtually free groups can be characterized as the quasi-geodesic submonoids. We use the fact that Cayley graphs of virtually free groups are polyhyperbolic geodesic metric spaces.

In Section 5 we show that being graded is decidable for finitely generated submonoids of vitually free groups. This is achieved by reducing the problem to test the number of factorizations with nontrivial factors for finitely many elements of the submonoid. Then we apply Muller-Schupp Theorem \cite{MS} and use classical properties of context-free languages to produce an algorithm.

In Section 6 we show that the word problem of a graded submonoid of a virtually free group (seen as a relation on the free monoid on the generating set $S$) is a rational subset of $S^* \times S^*$. This section features the key construction of the paper: for a given monoid in our class we effectively construct a finite automaton which accepts in a precise sense its word problem (the relations of the monoid). The strategy is to use the quasi-geodesic characterization to somehow navigate in a bounded subgraph of the Cayley graph of the ambient group.

Together with an additional property of rational subsets of the direct square of a group with decidable word problem, this construction allows us to prove in Section 7 that the homomorphism and isomorphism problems are decidable for graded submonoids of  virtually free groups. 

In Section 8 we prove that being graded, rational and Kleene is equivalent for submonoids of free groups, using standard language-theoretic tools such as the lexicographic order or syntactic congruences. We also discuss how these concepts relate in the general virtually free case.

Finally, we collect in Section 9 a few open problems.

%%%%%%%%%%%%

%We are interested in studying finitely generated submonoids of free groups. 
%We will denote finitely generated by f.g. for short. Note that f.g. submonoids are rational subsets, in particular, the membership problem is solvable for f.g. submonoids of free groups by Benois theorem \cite[Theorem II.6.1]{Sak}. All the homomorphisms, epimorphisms and isomorphisms in this paper are meant to be monoid homomorphisms, epimorphisms and isomorphisms respectively, unless stated otherwise.
\par

\section{Preliminaries}
We will denote finitely generated by f.g. for short. %Note that f.g. submonoids are rational subsets, in particular, the membership problem is solvable for f.g. submonoids of free groups by Benois theorem \cite[Theorem II.6.1]{Sak}. 
All the homomorphisms, epimorphisms and isomorphisms in this paper are meant to be monoid homomorphisms, epimorphisms and isomorphisms respectively, unless stated otherwise. \par 

	We briefly recall some basic definitions from formal language theory and automata theory, see (for example) \cite{HU}, \cite{Sak} for more details. \par  
	\subsection{Rational languages and finite automata}
	Let $A$ be a finite set (called alphabet), and $A^*$ be the corresponding free monoid. Recall that a {\it language} over $A$ is any subset of $A^*$, and the collection of {\it rational languages} (also known as regular languages) over $A$ consists of all the languages which can be obtained from finite languages by taking unions of two languages, concatenations of two languages and Kleene star of a language (i.e., the submonoid generated by this language). Kleene's theorem states that a language is rational if and only if it is accepted by some finite automaton. Recall that a (non-deterministic) {\it automaton} $\mathcal{A}$ over a finite alphabet $A$ is a 5-tuple $(Q, A, I, T, E)$, where
	\begin{itemize}
		\item $Q$ is a non-empty set, called the set of {\it states};
		\item $A$ is a finite non-empty set, called the {\it (input) alphabet};
		\item $I$ is a subset of $Q$, called the set of {\it initial states};
		\item $T$ is a subset of $Q$, called the set of {\it terminal states}; 
		\item $E$ is a subset of $Q \times A \times Q$, called the set of {\it transitions}.
	\end{itemize}
An automaton $\mathcal{A}$ is called:
\begin{itemize}
\item
{\it finite} if the sets $Q$ and $E$ are finite; % In this paper $A$ will always be a finite alphabet.
\item
{\it deterministic} if $|I|=1$ and $(p,a,q), (p,a,q') \in E$ implies $q=q'$;
\item
{\it complete} if $E \cap (\{ q \} \times \{ a \} \times Q) \neq \emptyset$ for all $q \in Q$ and $a \in A$.
\end{itemize}
Note  that every automaton $\mathcal{A}$ can be represented in the natural way as a digraph with edges labelled by the elements of $A$: the vertex set is $Q$, each transition $(q_1, a, q_2)$ in $E$ gives rise to an edge labelled by $a$ from $q_1$ to $q_2$, and there are two distinguished sets of vertices -- initial and terminal. Thus we can speak about vertices and (oriented) edges of $\mathcal{A}$, about spanning trees in $\mathcal{A}$, etc.

%Recall that, given a monoid $G$ (in our case $G$ will be a group), one can define rational subsets of $G$ in different equivalent ways, in particular, by using (non-deterministic) $G$-automata (or automata over $G$), see \cite{Sak}.
 %A $G$-automaton $\mathcal{A}$ consists of a set of states $Q$, a set of transitions, which is a subset of $Q \times G \times Q$, a set of initial states $I$ and a set of final states $F$.
  A {\it path} $p$ in $\mathcal{A}$  is a sequence $p=q_1a_1q_2a_2 \ldots q_n a_n q_{n+1}$, where $(q_i,a_i,q_{i+1})$ is a transition for all $i=1, \ldots, n$; the {\it label} of this path is defined to be $a_1a_2 \ldots a_n \in A^*$, and for $n=0$ the label of the trivial path is defined to be the empty word. 
We denote by $p \mapright{u} q$ a path from $p$ to $q$ with label $u$.  
A path $p$ is {\it successful} if $q_1 \in I$ and $q_{n+1} \in T$, and the language over $A$ {\it accepted} by $\mathcal{A}$, denoted by $L(\mathcal{A})$, is defined to be the set of labels of all successful paths in $\mathcal{A}$. 
%We will also write $p=e_1e_2 \ldots e_n$, where $e_i$ are the edges corresponding to the transitions $(q_i,a_i,q_{i+1})$; then $l(p)=l(e_1)l(e_2) \ldots l(e_n)$.
   Note that we can always suppose an automaton to be {\it trim}, i.e., such that every vertex belongs to some successful path, without changing the accepted language. We can also always suppose there is only one initial state. See \cite[Section I.4]{Ber}.
   \par 
    Note that any subgraph $C$ of $\mathcal{A}$ determines another automaton with the set of states consisting of the vertices of $C$, by restriction of transitions to the edge set of $C$, and restriction of the initial and terminal states to those which belong to $C$; we call this automaton a {\it subautomaton} of $\mathcal{A}$. If $\mathcal{B}$ is a subautomaton of $\mathcal{A}$, then it is immediate that $L(\mathcal{B}) \subseteq L(\mathcal{A})$, since every successful path in $\mathcal{B}$ gives rise to a successful path in $\mathcal{A}$ with the same label. 	
\par 
	%Finite and pushdown automata? \\
	%Rational subsets in groups and monoids.
	
	\subsection{Rational and recognizable subsets}
	
	Suppose $M$ is a finitely generated monoid. In analogy to rational languages in a free monoid, one can define the set of {\it rational subsets} in $M$ to consist of all the subsets of $M$ which can be obtained from the finite ones by taking unions of two subsets, products of two subsets and Kleene star of a subset (i.e., passing to the submonoid generated by the given subset). In particular, rational subsets of free monoids are precisely the rational languages, and a subgroup of a group is a rational subset if and only if it is finitely generated (due to Anisimov and Seifert \cite[Theorem III.2.7]{Ber}). Let $\alpha: A^* \rightarrow M$ be a surjective homomorphism, for some finite alphabet $A$. It turns out that a subset $K$ of $M$ is rational if and only if there exists a rational language $L$ in $A^*$ such that $K=\alpha(L)$, see \cite[Proposition 1.7, p.223]{Sak}; this property is sometimes taken as the definition of a rational subset.
	\par 
	 We say that $X \subseteq M$ is a {\em recognizable subset} of $M$ if there exists a homomorphism $\theta:M \to K$ to some finite monoid $K$ satisfying $X = \theta^{-1}\theta(X)$.  Given a surjective homomorphism  $\alpha: A^* \rightarrow M$, it turns out that a subset $X$ of $M$ is recognizable if and only if $\alpha^{-1}(X)$ is a rational language \cite[Theorem 2.2, p.247]{Sak}. % In particular, all the recognizable subsets of a f.g. monoid are rational. %In finitely generated free monoids the sets of rational and recognizable subsets coincide (due to Kleene). 
	 A subgroup of a group is a recognizable subset if and only if it has finite index \cite[Proposition 6.1, p.302]{Sak}. \par 
	 Equivalently, a subset $X$ of $M$ is recognizable if and only if the {\em syntactic congruence} $\sim_X$ has finite index (i.e., there are finitely many equivalence classes); this is the congruence on $M$ defined by $u \sim_X v$ if
$$\forall p,q \in M\, (puq \in X \Leftrightarrow pvq \in X);$$ 
see \cite[Theorem 2.3, p.247]{Sak}.
 	Let ${\rm Rat}(M)$ (respectively ${\rm Rec}(M)$) denote the set of all rational (respectively recognizable) subsets of $M$. For a finitely generated monoid $M$ we always have ${\rm Rec}(M) \subseteq {\rm Rat}(M)$ \cite[Proposition III.2.4]{Ber}. We say that $M$ is a {\em Kleene monoid} if ${\rm Rec}(M) = {\rm Rat}(M)$. This terminology arises from Kleene's Theorem stating that ${\rm Rec}(A^*) = {\rm Rat}(A^*)$ whenever $A$ is finite \cite[Theorem III.2.1]{Ber}, i.e., all rational languages are recognizable.

\section{Graded monoids and irreducible elements}

We now introduce graded monoids and some of their basic properties, following \cite{MMS}. \par 
 %\vspace{0.1 cm}
\begin{definition}\cite[Definition 1.3]{MMS} 
	Let $M$ be a monoid. $M$ is called {\it graded} if it has a finite system of generators $S$ such that every element of $M$ can be written as a word over $S$ only in finitely many different ways. In other words, $M$ is graded if there exists a finite alphabet $X$ and an epimorphism $\alpha: X^* \rightarrow M$ such that $\alpha^{-1}(g)$ is finite for every $g \in M$. In particular, every graded monoid is finitely generated. \par 
\end{definition}
	%\vspace{0.1 cm}
	
\begin{definition}\cite[Definition 1.5]{MMS}  A non-identity element of a monoid $M$ is called {\it irreducible} if it cannot be written as a product of non-identity elements of $M$. It follows immediately from the definition that any generating set of $M$ should contain all the irreducible elements of $M$, and that the set of irreducible elements of $M$ is equal to $(M \backslash \{ 1 \}) \backslash (M \backslash \{ 1 \})^2$. Given a f.g. submonoid $M$ of a virtually free group, one can compute the set of all irreducible elements of $M$ since  $(M \backslash \{ 1 \}) \backslash (M \backslash \{ 1 \})^2$ is an effectively constructible rational set.
\end{definition}
\par 
	
We summarize some known properties of graded monoids in the following lemma. The first statement basically means that the property of being graded does not depend on the choice of a finite generating set, provided this set does not contain the identity (note that no monoid can be graded with respect to a generating set containing the identity).	

\begin{lemma}\label{graded_properties}
Let $M$ be a graded monoid.
\begin{enumerate}
	\item If $S'$ is an arbitrary generating set of $M$ which does not contain the identity then 
	every element of $M$ can be written as a word over $S'$ only in finitely many ways.  
	\item $M$ has no (left, right) invertible elements apart from the identity.
	\item The set of irreducible elements of $M$ generates $M$.

\end{enumerate}
\end{lemma}

\begin{proof}
The proof can be found in \cite{MMS}.
\end{proof}

Note that all f.g. submonoids of a free monoid are graded. In particular, all f.g. free monoids are graded. \par 
\begin{example}\label{ex1} There exists a graded submonoid of a free group which cannot be embedded in a free monoid. Consider, for example, the submonoid
 $M$ of $F(a,b)$ generated by $a, b, a^{-1} b^{-1} ab$. It does not embed into a free monoid because it satisfies the equality $yxz = xy$, where $x,y,z$ are the three generators $a,b, a^{-1} b^{-1} ab$ respectively, and $z \neq 1$.
 \par 
We now show that $M$ is graded. We refer to a word from $\{a,b\}^*$ as positive. 
We can use induction on the number of generators in the decomposition of an element in $M$ to see that the reduced form of any word in $M$ must either be a positive word or end in $b^{-1} a^kbu$, where $k > 0$ and $u$ is a positive word. Let $L$ denote the set of all elements of $F$ of this form: either positive or ending in $b^{-1} a^kbu$.
Then it is easy to see that for every element $g$ in $L$ and any generators $x_i,x_j$ of $M$ from the above generating set (possibly equal) the length of $gx_i$ in $F$ is not smaller than the length of $g$, and the length of $gx_ix_j$ in $F$ is strictly bigger than the length of $g$. Thus, if $h \in M$ and $h=x_1 \ldots x_n$ is a decomposition of $h$ as a product of the above generators, then the length of $h$ is at least $[n/2]$. This implies that if $h$ has length $s$ in $F$, then any decomposition of $h$ has at most $2s+1$ factors, in particular, there are finitely many such decompositions, so $M$ is graded.

\end{example}

\par 

The following example shows that conditions 2 and 3 in Lemma \ref{graded_properties} are necessary, but not sufficient for implying graded.

\begin{example}\label{ex2} There exists a f.g. submonoid $M$ of a free group which has the following properties: 
	\begin{enumerate}
		\item $M$ is not graded.
		\item $M$ has no (left, right) invertible elements apart from the identity. %The identity is the only (left, right) invertible element of $M$;
		\item The set of irreducible elements of $M$ generates $M$. %$M$ is generated by the set of its irreducible elements.
	\end{enumerate}	
Consider, for example, the submonoid $M$ of the free group of rank 3, $F(a,b,c)$, generated by the elements $ba, c, c^{-1}a^{-1}, b^{-1}a^{-1}$. Note that $M$ is not graded, since $ba \cdot c \cdot c^{-1}a^{-1} \cdot b^{-1}a^{-1} = c \cdot c^{-1}a^{-1}$, and so $a^{-1}$ can be written as a product of generators of $M$ in infinitely many different ways: 
$$a^{-1}= (ba)^k \cdot c \cdot c^{-1}a^{-1} \cdot (b^{-1}a^{-1})^k, \quad k \geq 0.$$
\par

 We show now that $M$ has no invertible elements apart from the identity. Indeed, the product of two non-invertible elements in $M$ (or any other submonoid of a free group) is non-invertible: if $x, y \in M$, and $z=(xy)^{-1}=y^{-1}x^{-1} \in M$, then also $x^{-1}, y^{-1} \in M$. This means it suffices to prove that the generators  $ba, c, c^{-1}a^{-1}, b^{-1}a^{-1}$ of $M$ are non-invertible in $M$, i.e., the elements $a^{-1}b^{-1}, c^{-1}, ac, ab$ are not in $M$. Suppose that $a^{-1}b^{-1} \in M$, and $q,r,s,t$ are the numbers of times each of the generators $ba, c, c^{-1}a^{-1}, b^{-1}a^{-1}$ respectively occurs in some factorization of $a^{-1}b^{-1}$. Then $r=s$, since the sum of exponents of $c$'s should be equal to $0$, $t=q+1$, since the sum of exponents of $b$'s should be equal to $-1$, and  $s+t=q+1$, since the sum of exponents of $a$'s should be equal to $-1$. It follows that $r=s=0$.  But any product with factors in $\{ ba, b^{-1} a^{-1}\}$ is necessarily reduced, so we can never obtain $a^{-1}b^{-1}$, a contradiction, so $a^{-1}b^{-1} \notin M$. The other cases are similar. Note that invertible equals left (right) invertible since $M$ embeds in a group.
 \par

It remains to prove the third claim. Indeed, we claim that all the elements  $ba, c, c^{-1}a^{-1}, b^{-1}a^{-1}$ are irreducible.
Note that for each of these generators the first letter occurs only in this generator, but not in the rest (here we consider $a$ and $a^{-1}$ as different letters). Suppose  $ba=y_1y_2 \ldots y_m$ for some generators $y_1, \ldots, y_m$ from the above generating set. Since $b$ does not occur in any of the generators apart from $ba$, at least one of the generators $y_1, \ldots, y_m$ should be $ba$, and we can suppose that $y_i=ba$ and there is a reduction of the word $y_1 \ldots y_m$ to the reduced word $ba$ such that the letter $b$ from $y_i$ does not get cancelled, for some $1 \leq i \leq m$. But then $y_1 \ldots y_{i-1}=1$, and since $M$ has no non-trivial invertible elements, as proved above, we have $i=1$. Now we have $y_2 \ldots y_m=1$, so $m=1$. This shows that $ba$ is irreducible. The other cases are similar. Thus $M$ is generated by the set of its irreducible elements.

\end{example}

 \par 
 %\vspace{0.1 cm}
 
 The next example shows that conditions 2 and 3 in Lemma \ref{graded_properties} are not equivalent either.
   
\begin{example}\label{ex3} There exists a f.g. submonoid $M$ of a free group such that both claims hold for $M$:
\begin{enumerate}
\item $M$ has no (left, right) invertible elements apart from the identity.
%The identity is the only invertible element of $M$;
\item The set of irreducible elements of $M$ does not generate $M$. %$M$ is not generated by the set of its irreducible elements.
\end{enumerate}
For example, take $M$ to be the submonoid of the free group $F(a,b)$ generated by $a, ba^{-1}, b^{-1} a$. To prove the first claim, we only need to show that none of the generators is invertible, since the product of non-invertible elements is non-invertible, as in the previous example. That is, we need to check that $a^{-1}, ab^{-1}, a^{-1} b \notin M$. Suppose $ab^{-1} \in M$. Let $r,s,t$ be the number of times each of the generators $a, ba^{-1}, b^{-1} a$ respectively occurs in some particular factorization of $ab^{-1}$. Then $t = s+1$, since the sum of exponents of $b$'s should be equal to $-1$. It follows  that $r = 0$, since the sum of exponents of $a$'s should be equal to $1$. But any product with factors in $\{ ba^{-1}, b^{-1} a\}$ is necessarily reduced, so we can never obtain $ab^{-1}$, a contradiction, so $ab^{-1} \notin M$. The other cases are similar. Thus $M$ has no invertible elements apart from the identity.
\par 
%Now $a$ is not irreducible since $a = (ba^{-1})a(b^{-1} a)$.  Since $a$ cannot be written as a product of the other generators, it must belong to any generating set.
Now we show that $M$ is not generated by the set of its irreducible elements. Indeed, the set of irreducible elements of $M$ must be contained in any generating set of $M$, in particular, in $\{a, ba^{-1}, b^{-1} a \}$. However, $a$ is not irreducible in $M$, since $a= ba^{-1} \cdot a \cdot b^{-1}a$, and the set $\{ ba^{-1}, b^{-1}a \}$ clearly does not generate $M$, since any product with factors in $\{ ba^{-1}, b^{-1} a\}$ is necessarily reduced and so cannot be equal to $a$. This shows that both claims hold for $M$.

\end{example}

\par 
%\vspace{0.1 cm}
\begin{example}\label{ex4}
There exists a graded submonoid of a free group which is not finitely presented. Indeed, every f.g. submonoid of a free monoid is graded and can be embedded in a free group, but not every f.g. submonoid of a free monoid is finitely presented.

 The following example is well-known, we include it for the sake of completeness. Consider the submonoid $M$ of the free monoid $\{a,b,c,d\}^*$ generated by $ab,ad,ba,c,ca,d$ (denoted by $x_1,\ldots,x_6$ respectively). A relation $x_{i_1}\ldots x_{i_m} = x_{j_1}\ldots x_{j_n}$ holding in $M$ is called minimal if $i_1 \neq j_1$ and $x_{i_1}\ldots x_{i_r} = x_{j_1}\ldots x_{j_s}$ holding in $M$ with $0 \leq r \leq m$ and $0 \leq s \leq n$ implies either $r = s = 0$ or $r = m$, $s = n$. Since free monoids are cancellative, a presentation of $M$ can be obtained by considering all minimal relations on $x_1,\ldots,x_6$. It is easy to see that these are the relations of the form $x_4 x_1^k x_2 = x_5 x_3^k x_6$ for $k \geq 0$, hence $M$ is presented by
$$Mon\langle x_1,\ldots,x_6 \mid x_4 x_1^k x_2 = x_5 x_3^k x_6, \: k \geq 0 \rangle.$$
If $M$ were finitely presented, it would admit a presentation of the form
$$Mon\langle x_1,\ldots,x_6 \mid x_4 x_1^k x_2 = x_5 x_3^k x_6, \: 0 \leq k \leq t \rangle.$$
However, it is immediate that $x_4 x_1^{t+1} x_2 = x_5 x_3^{t+1} x_6$ cannot be derived from this finite set of relations. Thus $M$ is not finitely presented.

\end{example}

\begin{example}\label{ex5}
Not every graded submonoid of a virtually free group embeds into a free group. Indeed, take $G = \langle a,b \: | \: b^2 = 1, ab = ba \rangle$ and let $M$ be the submonoid of $G$ generated by $a$ and $ab$. Then it is easy to see that $M$ is graded, but does not embed into a free group since $a$ and $ab$ commute and have the same square.
\end{example}

In \cite{MMS}, the authors provide several alternative characterizations for graded monoids. We add a few, and in order to do so we need to introduce a few semigroup-theoretic concepts.

Let $M$ be a monoid and $u \in M$. A {\em nontrivial factorization} of $u$ is a sequence $(u_1,\ldots,u_n)$ of elements of $M \setminus \{ 1\}$ (for some $n \geq 0$) such that $u = u_1\ldots u_n$. The element $u$ is {\em regular} (respectively {\em idempotent}) if $u \in uMu$ (respectively $u^2 = u$). If $MuM \subseteq MvM$ (equivalently, if $u \in MvM$), we say that $v \in M$ is a {\em factor} of $u$ (and we write $u \leq_{\cal{J}} v$). The quasi-order $\leq_{\cal{J}}$ is called in the literature the ${\cal{J}}$-{\em order} of $M$. A monoid is {\em finite} $\cal{J}$-{\em above} if every element of $M$ has only finitely many factors.

\begin{proposition}

\label{eqgrad}

The following conditions are equivalent for a f.g. monoid $M$:

\begin{enumerate}

\item $M$ is graded.

\item Every element of $M$ has only finitely many nontrivial factorizations.

\item $M$ is finite $\cal{J}$-above and $1$ is the unique regular element of $M$.

\item $M$ is finite $\cal{J}$-above, and $1$ is the unique invertible element and the unique idempotent of $M$.

\end{enumerate}

\end{proposition}

\begin{proof}

$1 \Rightarrow 2$. Let $X$ be a finite alphabet and let $\varphi:X^* \to M$ be a surjective monoid homomorphism such that $\varphi^{-1}(u)$ is finite for every $u \in M$.

Suppose first that there exist arbitrarily long nontrivial factorizations $(u_1,\ldots,u_n)$ of $u$. Take $v_i \in \varphi^{-1}(u_i)$ for every $i$. Since $u_i \neq 1$, it follows that $v_1\ldots v_n \in \varphi^{-1}(u)$ and $|v_1\ldots v_n| \geq n$, thus $\varphi^{-1}(u)$ would be infinite, a contradiction.

Therefore there exists a constant $K \in \mathbb{N}$ such that every nontrivial factorization $(u_1,\ldots,u_n)$ of $u$ satisfies $n \leq K$. Let $U$ be the set of all elements of $M$ occurring in some nontrivial factorization of $u$.

Suppose that $U$ is infinite. Then $\varphi^{-1}(U)$ contains arbitrarily long words and so does $\varphi^{-1}(u)$, a contradiction. Hence $U$ is finite and in view of the bound $K$ we deduce that $u$ has only finitely many nontrivial factorizations.

$2 \Rightarrow 1$. Straightforward.

$2 \Rightarrow 3$. Every factor of $u \in M$ is either the identity or occurs in some nontrivial factorization of $u$. Therefore $u$ has only finitely many factors and $M$ is finite $\cal{J}$-above.

Suppose that $u \in M\setminus \{ 1\}$ is regular. Then $u = uvu$ for some $v \in M$. If $v = 1$ (respectively $v \neq 1$), then $(u,u,\ldots,u)$ (respectively $(u,v,u,\ldots,v,u)$) provides arbitrarily long nontrivial factorizations of $u$, a contradiction. Thus $1$ is the unique regular element of $M$.

$3 \Rightarrow 4$. Since idempotents and invertible elements are regular.

$4 \Rightarrow 3$. Suppose that $u \in M$ is regular. Then $uvu = u$ for some $v \in M$. Since $uv$ and $vu$ are both idempotents, it follows that $uv = vu = 1$. But then $u$ is invertible and so $u = 1$.

$3 \Rightarrow 2$. Let $u \in M$ have $n$ factors. Suppose that $u$ admits infinitely many factorizations. Then $u$ admits a nontrivial factorization $(u_1,\ldots,u_m)$ of length $m > n$. It follows that $u_1\ldots u_j$ is a factor of $u$ for $j = 1,\ldots,n+1$. Since $u$ has $n$ factors in $M$, it follows that there exist $1 \leq p < q \leq n+1$ such that $u_1\ldots u_p = u_1\ldots u_q$. Write $x = u_1\ldots u_p$ and $y = u_{p+1}\ldots u_q$. Then $xy = x$.

Suppose first that $y \neq 1$. Since $xy = x$ implies $xy^k = x$ for every $k \geq 1$, then $y^k$ is a factor of $x$ for every $k$. Since $M$ is finite $\cal{J}$-above, we get $y^r = y^s$ for some $1 \leq r < s$. Hence $y^r = y^{r +(s-r)} = y^{r +(s-r)r} = y^r y^{(s-r-1)r}y^r$ and so $y^r$ is regular. It follows that $y^r = 1$ and so $y = yy^{r-1}y$. Therefore $y \neq 1$ is regular, a contradiction.

Assume now that $y = 1$. Then $u_q = u_q(u_{p+1}\ldots u_{q-1})u_q \neq 1$ is regular, also a contradiction. Therefore $u$ admits only finitely many factorizations.

\end{proof}

Note that none of the two last conditions in 4 can be removed: for counterexamples, we can take any finite nontrivial group and any finite nontrivial semilattice (commutative monoid consisting of idempotents).

We can now apply Proposition \ref{eqgrad} to the case of submonoids of a free group.

\begin{corollary}

\label{eqgradfree}

The following conditions are equivalent for a f.g. submonoid $M$ of a free group:

\begin{enumerate}

\item $M$ is graded.

\item $M$ is finite $\cal{J}$-above.

\end{enumerate}

\end{corollary}

\begin{proof}

%%%%%%%%%%%% Changed
%Indeed, if $M$ is graded, then $M$ is finite $\cal{J}$-above. 
Suppose that $M$ is finite $\cal{J}$-above. Since $M$ is a submonoid of a group, $1$ is the unique idempotent of $M$, and since free groups are torsion-free, the fact that $M$ is finite $\cal{J}$-above implies that $1$ is the unique invertible element of $M$.
Now the claim follows immediately from the equivalence of the first and fourth conditions in Proposition \ref{eqgrad}. 

%The equivalence follows from Proposition \ref{eqgrad} by noting that $M$ is finite $\cal{J}$-above implies that $1$ is the unique invertible element and the unique idempotent of $M$: the second claim is trivial since $M$ embeds in a group and the first follows from free groups being torsion-free.

%%%%%%%%%%%%%%%

\end{proof}

We remark that we cannot remove f.g. from the statement of the corollary since a free monoid of infinite rank embeds in any free group of rank $> 1$.

\vspace{0.2 cm}

\section{A geometric characterization}

We provide in this section a geometric characterization of graded submonoids of virtually free groups using the concept of quasi-geodesic in the Cayley graph of the ambient group. We start by recalling some concepts from the theory of hyperbolic groups.

A {\em geodesic metric space} is a metric space $(X,d)$ such that, for all $x,y \in X$, there exists an isometry $\lambda$ from $[0,d(x,y)] \subset \mathbb{R}$ into $X$ such that $\lambda(0) = x$ and $\lambda(d(x,y)) = y$. Then we say that $\lambda$ is a {\em geodesic} from $x$ to $y$. We also call its image a geodesic, often denoted by $[x,y]$. 

A {\em geodesic} $p$-gon is a union of geodesics
$$[x_0,x_1] \cup [x_1,x_2] \cup \ldots \cup [x_{p-1},x_p]$$
in $(X,d)$ with $x_p = x_0$. This geodesic polygon is $K$-{\em thin} if for every side of it every point on this side is at distance $\leq K$ from some point on one of the other sides.

Then $(X,d)$ is {\em hyperbolic} if there exists some constant $K \geq 0$ such that every geodesic triangle in $(X,d)$ is $K$-thin. We say that $(X,d)$ is {\em polyhyperbolic} if there exists some constant $K \geq 0$ such that, for every $p \geq 2$, every geodesic $p$-gon in $(X,d)$ is $K$-thin. 

Let $[a,b] \subset \mathbb{R}$ and $\lambda \geq 1$, $\varepsilon \geq 0$. A mapping $\varphi:[a,b] \to X$ is a $(\lambda,\varepsilon)$-{\em quasi-geodesic} if
$$\frac{1}{\lambda}|x-y| - \varepsilon \leq d(\varphi(x),\varphi(y)) \leq \lambda |x-y| + \varepsilon \mbox{ holds for all }x,y \in [a,b].$$

Now let $Cay_A(G)$ be the Cayley graph of a f.g. group $G$ with respect to a fixed finite generating set $A$. The {\em geodesic distance} $d_A$ between vertices $g,h \in G$  is defined as the length of the shortest path connecting them. Since $d_A$ takes values in $\mathbb{N}$, $(G,d_A)$ is not a geodesic metric
space. However, we can remedy that by embedding $(G,d_A)$
isometrically into the {\em geometric 
realization} $\overline{Cay}_A(G)$ of $Cay_A(G)$, where vertices become
points and edges 
become segments of length 1 in some (euclidean) space, intersections
being determined by adjacency only. With the obvious metric,
$\overline{Cay}_A(G)$ is a geodesic metric space, and the geometric 
realization is unique up to isometry. We denote also by $d_A$ the
induced metric on $\overline{Cay}_A(G)$.

By \cite[Theorem 4.1]{AS}, $\overline{Cay}_A(G)$ is {\em polyhyperbolic} if and only if $G$ is virtually free. 

Suppose now that $M$ is a submonoid of $G$ generated by a finite set $S \subseteq G \setminus \{ 1 \}$. From this point onwards, we denote by $S^*$ the free monoid on $S$ (viewed as an alphabet) and not the submonoid of $G$ generated by $S$ (which is denoted by $M$). We denote by $Cont(S,G,A)$ the set of all mappings of the following form: for some $v = v_1\ldots v_n \in S^*$ $(v_i \in S)$, $\varphi_v:[0,n] \to \overline{Cay}_A(G)$ is a mapping such that:
\begin{itemize}
\item
$\varphi_v(i) = v_1\ldots v_i$ for $i = 1,2,\ldots,n$,  $\varphi_v(0)=1$;
\item
the interval $[i-1,i]$ is mapped homeomorphically into a geodesic $[\varphi_v(i-1),\varphi_v(i)]$ for $i = 1,2,\ldots,n$.
\end{itemize}
We call the elements of $Cont(S,G,A)$ {\em continuous paths} from $S^*$ in $\overline{Cay}_A(G)$.

We say that $M$ is a {\em quasi-geodesic} submonoid of $G$ with respect to $A$ if there exist $\lambda \geq 1$ and $\varepsilon \geq 0$ such that $\varphi$ is a $(\lambda,\varepsilon)$-quasi-geodesic of $(\overline{Cay}_A(G),d_A)$ for every $\varphi \in Cont(S,G,A)$.

Finally, we introduce some helpful mappings. Given a finite set of generators $S$ of a graded monoid $M$ and $u \in M$, write
$$\xi_S(u) = \cup_{n \geq 1} \{ (s_1,\ldots, s_n) \in S^n \mid s_1\ldots s_n = u%\mbox{ holds in }M
\}, \quad \Xi_S(u) = \max\{ n \geq 1 \mid \xi_S(u) \cap S^n \neq \emptyset \}.$$
If $A$ is a finite generating set for the ambient group $G$, we define also a mapping $\zeta_{S,A}:\mathbb{N} \to \mathbb{N}$ by
$$\zeta_{S,A}(n) = \max\{ \Xi_S(u) \mid u \in M,\, d_A(1,u) \leq n\}.$$

\begin{theorem}
\label{qg}
Let $G$ be a virtually free group generated by a finite set $A$. Let $M$ be a submonoid of $G$ generated by a finite set $S \subseteq G \setminus \{ 1 \}$. 
Then the following conditions are equivalent:
\begin{itemize}
\item[(i)] $M$ is graded;
\item[(ii)] $M$ is a quasi-geodesic submonoid of $G$ with respect to $A$.
\end{itemize}
\end{theorem}

\begin{proof}
(i) $\Rightarrow$ (ii). Let $K$ be a polyhyperbolicity constant for $(\overline{Cay}_A(G),d_A)$.
Let $\varphi \in Cont(S,G,A)$ be defined as above for $v = v_1\ldots v_n \in S^*$. Write $w_k = v_1\ldots v_k$ for $k = 0,1,\ldots,n$.
Let
\begin{itemize}
\item
$L = \max\{ d_A(1,s) \mid s \in S \}$,
\item
$L' = \zeta_{S,A}(2K+L+1)$,
\item
$\lambda = \max\{ L,L'\}$,
\item
$\varepsilon = \max\{ 3L,  2L + \frac{1}{L'} \}$.
\end{itemize}

Suppose that $i,j \in \{ 0,1,\ldots,n \}$ with $i < j$. Consider a geodesic polygon in $\overline{Cay}_A(G)$ of the form
$$[w_i,w_{i+1}] \cup [w_{i+1},w_{i+2}] \cup \ldots \cup [w_{j-1},w_{j}] \cup [w_j,w_i].$$
Assume that $[w_j,w_i]$ is the geodesic
$$w_j = q_0 \mapright{a_1} q_1 \mapright{a_2} \ldots \mapright{a_r} q_r = w_i$$
with $a_1,\ldots, a_r \in A \cup A^{-1}$. Let $\ell \in \{ 1,\ldots,r-1\}$. By definition of $K$, there exist $t \in \{ i,\ldots,j-1\}$ and $x \in [w_t,w_{t+1}]$ such that $d_A(q_{\ell},x) \leq K$. Since $d_A(w_t,w_{t+1}) = d_A(1,v_{t+1}) \leq L$, there exists some $z_{\ell} \in \{ i,\ldots,j\}$ such that $d_A(q_{\ell},w_{z_{\ell}}) \leq K+\frac{L}{2}$. Fix also $z_0 = j$ and $z_r = i$. We have 
$$d_A(w_{z_{\ell-1}},w_{z_{\ell}}) \leq d_A(w_{z_{\ell-1}},q_{\ell-1}) + d_A(q_{\ell-1},q_{\ell}) + d_A(q_{\ell},w_{z_{\ell}}) \leq 2K+L+1,$$
hence
%\begin{equation}
%\label{qg1}
$$|z_{\ell} - z_{\ell-1}| \leq \Xi_S(w_{z_{\ell-1}}^{-1} w_{z_{\ell}}) \leq L'$$
%\end{equation}
for $\ell = 1,\ldots,r$. It follows that 
\begin{equation}
\label{qg1}
|j-i| = |z_r - z_0| \leq \sum_{\ell = 1}^r |z_{\ell} - z_{\ell-1}| \leq L' r = L' d_A(w_j,w_i) = L' d_A(\varphi(j),\varphi(i)).
\end{equation}

On the other hand,
\begin{equation}
\label{qg2}
d_A(\varphi(j),\varphi(i)) = d_A(w_j,w_i) \leq \sum_{\ell = i+1}^{j} d_A(w_{\ell-1},w_{\ell}) = \sum_{\ell = i+1}^{j} d_A(1,v_{\ell})
\leq L|j-i|.
\end{equation}

Now, given $x,y \in [0,n]$, (\ref{qg2}) yields
\begin{equation}
\label{qg3}
\begin{array}{lll}
d_A(\varphi(x),\varphi(y))&\leq&d_A(\varphi(x),\varphi(\lfloor x\rfloor )) + d_A(\varphi(\lfloor x\rfloor ),\varphi(\lfloor y\rfloor )) + d_A(\varphi(\lfloor y\rfloor ),\varphi(y))\\
&\leq&L|\lfloor x\rfloor - \lfloor y\rfloor | +2L \leq L|x-y| + 3L.
\end{array}
\end{equation}

On the other hand, (\ref{qg1}) yields
$$\begin{array}{lll}
|x-y|&\leq&|\lfloor x\rfloor - \lfloor y\rfloor | +1 \leq L'd_A(\varphi(\lfloor x\rfloor ),\varphi(\lfloor y\rfloor )) +1\\
&\leq&L'(d_A(\varphi(\lfloor x\rfloor ),\varphi(x)) + d_A(\varphi(x),\varphi(y)) + d_A(\varphi(y),\varphi(\lfloor y\rfloor )))+1\\
&\leq&L'd_A(\varphi(x),\varphi(y)) + 2LL'+1.
\end{array}$$
and therefore
\begin{equation}
\label{qg4}
\frac{1}{L'}|x-y| - (2L + \frac{1}{L'}) \leq d_A(\varphi(x),\varphi(y)).
\end{equation}

By definition of $\lambda$ and $\varepsilon$, it follows from (\ref{qg3}) and (\ref{qg4}) that $\varphi$ is a $(\lambda,\varepsilon)$-quasi-geodesic of $(\overline{Cay}_A(G),d_A)$.

(ii) $\Rightarrow$ (i). If $M$ is not graded, there exists some $x \in M$ represented by infinitely many words $v \in S^*$. Since $S$ is finite, the length of such $v$ is unbounded, which prevents the existence of $\lambda, \varepsilon$ satisfying
$$\frac{1}{\lambda}\big{\lvert}\,|v|-0\,\big{\lvert} - \varepsilon \leq d_A(\varphi(|v|),\varphi(0)) = d_A(x,1)$$
for every $v$.

Note that this implication holds for every f.g. group $G$.
\end{proof}

It follows from Theorem \ref{qg} that, for a virtually free group $G$, the concept of quasi-geodesic submonoid is independent from the finite generating set of $G$ considered. Therefore we can speak of  quasi-geodesic submonoids of $G$.

\vspace{0.2 cm}

\section{Being graded is decidable in virtually free groups}

We now show that it is possible to algorithmically decide whether a f.g. submonoid of a virtually free group is graded or not.

We will need the following two lemmas.

\begin{lemma}\label{graded_bound}
Let $G$ be a virtually free group %with canonical basis $B = A \cup \{ b_1,\ldots,b_m \}$.
generated by the finite set $A$.
	Let $M$ be a f.g. submonoid of $G$, given by its finite generating set $S$. %, and $\alpha: B^* \rightarrow M$, $\alpha(b_i)=z_i$.
	 Then there  exists a computable constant $C > 0$ such that the following conditions are equivalent:
	 \begin{itemize}
	 \item[(i)] $M$ is graded.
	 \item[(ii)] $|\xi_S(g)| < \infty$ for every $g \in G$ such that $d_A(1,g) \leq C$.
	 \end{itemize}
\end{lemma}

\begin{proof}
By \cite[Theorem 4.1]{AS}, $Cay_A(G)$ is polyhyperbolic and it follows from the proof that a polyhyperbolicity constant $K$  for $Cay_A(G)$ can be computed from an effective description of $G$ (via quasi-isometry constants for standard constructions). Let 
$$L = \max\{ d_A(1,s) \mid s \in S \}\; \mbox{ and }\; C = 2K+L+1.$$
Since $G$ has decidable word problem, both $L$ and $C$ are computable.

(i) $\Rightarrow$ (ii). Trivial.

(ii) $\Rightarrow$ (i).
Suppose that $M$ is not graded. Then there exists some $g \in G$ such that $|\xi_S(g)| = \infty$. We may assume that $d_A(1,g)$ is minimum for this property. %Write $g = ub_k$ with $u \in F$ and $k \in \{ 0,\ldots,m\}$. 

Let $[g,1]$ be a geodesic in $Cay_A(G)$. Take $h \in [g,1]$ such that $d_A(1,h) = \lfloor \frac{d_A(1,g)}{2} \rfloor$. Let $D = \{ x \in G \mid d_A(x,h) \leq K + \frac{L}{2}\}$. Note that $A$ is finite, hence $D$ is finite as well. 

Let $g=z_{i1}z_{i2} \ldots z_{ij_i}$, $i \in \mathbb{N}$ denote infinitely many different decompositions of $g$, considered as equalities in the group $G$, where each $z_{ij} \in S$. Note that the set $\{ j_i, \: i \in \mathbb{N} \}$ is unbounded, since there are only finitely many different words of bounded length on $S$. We claim that
\begin{equation}
\label{poly1}
\mbox{for every $i \in \mathbb{N}$, $z_{i1}\ldots z_{ir_i} \in D$ for some $r_i \in \{ 0,\ldots,j_i\}$}.
\end{equation}
Indeed, fix a geodesic $[z_{i1}\ldots z_{ij}, z_{i1}\ldots z_{i,j+1}]$ for $j = 0,\ldots,j_i-1$, and consider the geodesic polygon 
$$[1, z_{i1}] \cup [z_{i1}, z_{i1}z_{i2}] \cup \ldots \cup [z_{i1}\ldots z_{ij_{i}-1}, z_{i1}\ldots z_{ij_i}] \cup [g, 1].$$
Since $K$ is a polyhyperbolicity constant for $Cay_A(G)$, there exist some $j \in \{ 0,\ldots,j_i-1 \}$ and $y \in [z_{i1}\ldots z_{ij}, z_{i1}\ldots z_{i,j+1}]$ such that $d_A(h,y) \leq K$. Now $d_A(z_{i1}\ldots z_{ij}, z_{i1}\ldots z_{i,j+1}) = d_A(1,z_{i,j+1}) \leq L$, so $d_A(y, z_{i1}\ldots z_{ij}) \leq \frac{L}{2}$ or $d_A(y, z_{i1}\ldots z_{i,j+1}) \leq \frac{L}{2}$. Thus $d_A(y,z_{i1}\ldots z_{ij}) \leq \frac{L}{2}$ for some $j \in \{ 0, \ldots, j_i\}$, yielding 
$$d_A(z_{i1}\ldots z_{ij},h) \leq d_A(z_{i1}\ldots z_{ij},y) + d_A(y,h) \leq \frac{L}{2} + K.$$
Thus (\ref{poly1}) holds.

For each $x \in D$, write
$$I_x = \{ i \in \mathbb{N} \mid z_{i1}\ldots z_{ir_i} = x\}.$$
Since $D$ is finite, it follows from (\ref{poly1}) that $I_x$ is infinite for some $x \in D$. 

Suppose first that
$R = \{ r_i \mid i \in X\}$ is infinite. Then $|\xi_S(x)| = \infty$. By minimality of $d_A(1,g)$, we get $d_A(1,x) \geq d_A(1,g)$. Hence
$$d_A(1,g) \leq d_A(1,x) \leq d_A(1,h) + d_A(h,x) \leq \frac{d_A(1,g)}{2} +K+\frac{L}{2}$$
and so $d_A(1,g) \leq 2K+L < C$, contradicting condition (i).

Thus $R$ is finite and so $|\xi_S(x^{-1}g)| = \infty$. By minimality of $d_A(1,g)$, we get $d_A(1,x^{-1}g) \geq d_A(1,g)$. Hence
$$\begin{array}{lll}
d_A(1,g)&\leq&d_A(1,x^{-1}g) = d_A(1,g^{-1}x) \leq d_A(1,g^{-1}h) + d_A(g^{-1}h,g^{-1}x)\\
&=&d_A(g,h) + d_A(h,x) \leq \frac{d_A(1,g)+1}{2} + K+\frac{L}{2}
\end{array}$$
and so $d_A(1,g) \leq 2K+L +1 = C$, contradicting condition (i) once again. Therefore $M$ is graded.
\end{proof}

The next lemma features a class of languages (containing rational languages as a particular case) known as {\em context-free languages}. They are usually defined through structures called {\em context-free grammars} or {\em pushdown automata}. For the basic theory, the reader is referred to \cite{Ber,HU}. We summarize next the properties of context-free languages which are relevant to us:

\begin{itemize}

\item[(CF1)]

If $L \subseteq A^*$ is context-free and $R \subseteq A^*$ is rational, then $L \cap R$ and 

$$R\diamond L = \{ u \in A^* \mid vu \in L \mbox{ for some }v \in R \}$$

are both context-free.

\item[(CF2)]

It is decidable whether or not a context-free language recognized by a given pushdown automaton is finite.

\end{itemize}

The celebrated Muller and Schupp Theorem illustrates the relevance of context-free languages for group theory: let $A$ be a finite alphabet, $G$ a group and $\varphi:(A \cup A^{-1})^* \to G$ a surjective monoid homomorphism satisfying $\varphi(a^{-1}) = (\varphi(a))^{-1}$ for every $a \in A$; then $G$ is virtually free if and only if $\varphi^{-1}(1)$ is context-free \cite{MS}.

We should remark that the closure properties (CF1) follow from effective algorithms: given a pushdown automaton recognizing $L$ and a finite automaton recognizing $R$, we can effectively construct pushdown automata recognizing $L \cap R$ and $R \diamond L$. Moreover, given a finite presentation for the virtually free group $G$ and $\varphi$, we can effectively construct a pushdown automaton recognizing $\varphi^{-1}(1)$.

\begin{lemma}\label{context-free}

Let $X$ be a finite alphabet and let $\alpha: X^* \rightarrow G$ be a homomorphism into a virtually free group $G$. Then for every $g \in G$ the language $\alpha^{-1}(g) \subseteq X^*$ is context-free, and given $\alpha$ and $g$ one can construct a pushdown automaton recognizing $\alpha^{-1}(g)$.

\end{lemma}

\begin{proof}

Let $X^{-1}$ denote a set of {\em formal inverses} for the set $X$. We extend $\alpha$ to a monoid homomorphism $\overline{\alpha}:(X \cup X^{-1})^* \to G$ by setting $\overline{\alpha}(x^{-1}) = (\alpha(x))^{-1}$ for every $x \in X$. Then the image of $\overline{\alpha}$ is a subgroup $H$ of $G$. If $F$ is a finite index free subgroup of $G$, then $F \cap H$ is a finite index subgroup of $H$. Moreover, $F \cap H$ is free by Nielsen's Theorem, hence $H$ is 
virtually free. By Muller and Schupp Theorem, $\overline{\alpha}^{-1}(1)$ is an (effectively constructible) context-free language. 

It is well known how to describe the rational subsets of $G$. If $G = Fb_0 \cup \ldots \cup Fb_m$ is a decomposition of $G$ as a union of disjoint right cosets, then Rat$(G)$ consists of all subsets of the form
$$L_0b_0 \cup \ldots \cup L_mb_m \mbox{ for }L_0,\ldots,L_m \in {\rm Rat}(F)$$
\cite{Gru} (see also \cite[Prop. 4.1]{Silva2}). Moreover, the languages $L_i$ can be effectively computed from any rational expression or finite automaton describing the rational subset. Thus it is decidable, for a given $g \in G$, whether or not $g \in H$. We may therefore assume that $g \in H$ (otherwise $\alpha^{-1}(g)$ is empty, hence context-free), and in that case we can compute some $w \in (X \cup X^{-1}) ^*$ such that $\overline{\alpha}(w) = g^{-1}$. It is straightforward to check that $\overline{\alpha}^{-1}(g) = w \diamond \overline{\alpha}^{-1}(1)$, hence context-free (and effectively constructible) in view of (CF1). Now it follows as well from (CF1) that $\alpha^{-1}(u) = \overline{\alpha}^{-1}(u) \cap X^*$ is also an (effectively constructible) context-free language. 

\end{proof}

\begin{theorem}\label{graded_decidable}
	Given a f.g. submonoid $M$ of a virtually free group $G$, it is decidable whether $M$ is graded or not.
\end{theorem}

\begin{proof}
	Since a subgroup of a virtually free group is necessarily virtually free, we may assume that $G$ is generated by some finite set $A$. Suppose $M$ is given by a finite generating set $S$ and let $\alpha: S^* \rightarrow G$ be the canonical homomorphism induced by $s \mapsto s$ $(s \in S)$. 
According to Lemma \ref{graded_bound}, we can compute $C>0$ such that it suffices to check whether $\alpha^{-1}(g)$ is finite for all $g \in G$ such that $d_A(1,g) \leq C$, i.e., for finitely many elements. Let $w_1, \ldots, w_k$ be all elements of $G$ of length at most $C$. For each $i$, $i=1, \ldots, k$, %first check whether $w_i$ is in $M$ (this can be done since membership problem is solvable for submonoids of free groups by Xenois theorem), and if yes, then 
	 $\alpha^{-1}(w_i)$ is context-free and given by its pushdown automaton by Lemma \ref{context-free}, so one can decide whether $\alpha^{-1}(w_i)$ is finite or not by (CF2). If for some $i=1, \ldots, k$ the language $\alpha^{-1}(w_i)$ is infinite then $M$ is not graded, otherwise it is.
\end{proof}

\section{Rational word problem}

In group theory, the expression {\rm word problem} is used with a double meaning: it may refer to the classical Dehn decidability problem, or to the set of words on a certain (finite) generating set representing the identity.

In the monoid case, we need a relation, since the congruence class of the identity does not determine at all a congruence on the free monoid. Thus let $M$ be a monoid generated by a finite set $S$, and let $\alpha:S^* \to M$ denote the canonical projection. The {\em word problem} of $M$ with respect to $S$ is defined as 
$$WP_S(M) = \{ (u,v) \in S^* \times S^* \mid \alpha(u) = \alpha(v) \}.$$
Clearly, the word problem (decidability problem) of $M$ is decidable if and only if $WP_S(M)$ is recursive (i.e. has decidable membership problem). We are particularly interested in the case of $WP_S(M)$ being rational.
By \cite[Theorem 4.2]{NPR}, $WP_S(M)$ is always rational when $M$ is finite. On the other hand, if $M$ is an infinite group, then $WP_S(M)$ is never rational \cite[Theorem 6.1]{NPR}. We will discuss the case of graded submonoids of virtually free groups.

We start by noting that the property of $WP_S(M)$ being rational is independent from the finite generating set $S$. % through two lemmas. 
In order to apply results from \cite{NPR}, we consider also semigroup word problems: if $S$ is a semigroup generating set for a semigroup $M$ and $\alpha:S^+ \to M$ denotes the canonical projection from the free semigroup $S^+ = S^* \setminus \{ 1 \}$ onto $M$, write
$$WP'_S(M) = \{ (u,v) \in S^+ \times S^+ \mid \alpha(u) = \alpha(v) \}.$$

Since $S^* \times S^*$ is generated by $\hat{S} = (S \times \{ 1 \}) \cup (\{ 1\} \times S)$, it is convenient to consider $\hat{S}$ as an alphabet. Then $\pi_1,\pi_2:\hat{S}^* \to S^*$ denote the projections on both components and $\pi:\hat{S}^* \to S^* \times S^*$ is the homomorphism defined by $\pi(u) = (\pi_1(u),\pi_2(u))$. It follows that Rat$(S^* \times S^*) = \{ \pi(L) \mid L \in {\rm Rat}(\hat{S}^*) \}$. A similar description holds for Rat$(S^+ \times S^+)$.

%\begin{lemma}
%\label{addone}
%Let $S$ be a finite monoid generating set for a monoid $M$ and let $M'$ be the subsemigroup of $M$ generated by $S$. Then $WP'_{S \cup \{ 1 \}}(M)$ is rational if and only if $WP'_{S}(M')$ is rational.
%\end{lemma}
%
%\begin{proof}
%We may assume that $1 \notin M'$, otherwise the claim follows immediately from \cite[Corollary 5.4]{NPR}. It follows that $1 \notin S$ either.
%
%Suppose that $WP'_{S}(M')$ is a rational subset of $S^+ \times S^+$. Then there exists a finite automaton $\cal{A} = (Q,\hat{S}, I,T,E)$ such that $WP'_{S}(M') = \pi(L({\cal{A}}))$. 
%
%
%\end{proof}

The following lemma also holds for generating sets containing 1, but since we do not care for those in the context of graded monoids, we just prove the simplest version.

\begin{lemma}
\label{gensets}
Let $S,S' \subseteq M \setminus \{ 1\}$ be finite generating sets for a monoid $M$. Then $WP_{S}(M)$ is rational if and only if $WP_{S'}(M)$ is rational.
\end{lemma}

\begin{proof}
Let $T$ denote the subsemigroup of $M$ generated by $S$. Since $1 \notin S$, we have
$$T = \left\{
\begin{array}{ll}
M&\mbox{ if }1 \in (M \setminus \{ 1\})^+\\
M\setminus\{ 1\}&\mbox{ otherwise}
\end{array}
\right.$$
Hence $T$ is also the subsemigroup of $M$ generated by $S'$. 

By \cite[Theorem 6.2]{NPR}, $WP_{S}(M)$ is rational if and only if $WP'_{S}(T)$ is rational. Similarly, $WP_{S'}(M)$ is rational if and only if $WP'_{S'}(T)$ is rational. Now the desired equivalence follows from \cite[Corollary 5.4]{NPR}, which establishes the analogous equivalence for semigroups.
\end{proof}

Let $S$ be a finite subset of a group $G$, generating a submonoid $M$.  Let $\alpha_S:S^* \to M$ be the canonical homomorphism.
%We define the alphabet $\hat{S} = (S \times \{ 1 \}) \cup (\{ 1\} \times S)$. Then $\pi_1,\pi_2:\hat{S}^* \to S^*$ denote the projections on both components and $\pi:\hat{S}^* \to S^* \times S^*$ is the homomorphism defined by $\pi(u) = (\pi_1(u),\pi_2(u))$.
We define an automaton $\Gamma_S$ over 
$\hat{S} = (S \times \{ 1 \}) \cup (\{ 1\} \times S)$ as follows:
\begin{itemize}
\item
the vertex set is the subset $M^{-1}M=\{x^{-1}y | \: x, y \in M \} \subseteq G$;
\item
$1$ is the only initial and terminal state;
\item
for every vertex $w \in M^{-1}M$ and $x \in \hat{S}$, there exists an edge 
$w \mapright{x} (\alpha_S\pi_1(x))^{-1}w(\alpha_S\pi_2(x))$.
\end{itemize}
%Let $L(\Gamma_S) \subseteq \hat{S}^*$ be the language accepted by $\Gamma_S$.
The next lemma shows that $\Gamma_S$ recognizes somehow the word problem of $M$ (i.e. the set of all relations in $M$). 

\begin{lemma}\label{language}
	Let $M$ be a submonoid of a group $G$, generated by a finite subset $S$. Then $WP_S(M) = \pi(L(\Gamma_S))$.
%Let $u_1, u_2 \in S^*$. Then $(u_1,u_2) \in \pi(L(\Gamma_S))$ if and only if $\alpha_S(u_1)=\alpha_S(u_2)$. 
\end{lemma}

\begin{proof}

Let $w = x_1\ldots x_n \in \hat{S}^*$, with $x_1,\ldots,x_n \in \hat{S}$. It follows from the definition of $\Gamma_S$ (which is deterministic and complete) that $w$ labels a unique path out of the basepoint 1, precisely
$$1 \mapright{w} (\alpha_S\pi_1(x_n))^{-1} \ldots (\alpha_S\pi_1(x_1))^{-1}(\alpha_S\pi_2(x_1))\ldots (\alpha_S\pi_2(x_n)) = (\alpha_S\pi_1(w))^{-1}(\alpha_S\pi_2(w)).$$
Hence 
$$w \in L(\Gamma_S)\; \mbox{ if and only if }\; \alpha_S\pi_1(w) = \alpha_S\pi_2(w) \; \mbox{ if and only if }\;  \pi(w) \in WP_S(M).$$
Since $\pi$ is surjective, the lemma follows.

\end{proof}

Assume now that $G$ is a virtually free group generated by a finite set $A$, and let $S$ be a finite subset of $G \setminus \{ 1\}$ generating a graded submonoid $M$.
Given $C > 0$, let $\Gamma_S^C$ denote the subautomaton of $\Gamma_S$ induced by the subset of vertices
$\{ g \in M^{-1}M \mid d_A(1,g) \leq C\}$. Since $A$ is finite, $\Gamma_S^C$ is a finite subautomaton of $\Gamma_S$. Since $G$ has decidable word problem, we can compute the finite set $\{ g \in G \mid d_A(1,g) \leq C\}$. Since $M^{-1}M$ is a rational subset of $G$, it has decidable membership problem, hence we can compute the vertex set of $\Gamma_S^C$ and therefore $\Gamma_S^C$ itself.

Our strategy is to show that we can compute $C$ such that $\Gamma_S^C$ may replace $\Gamma_S$ in Lemma \ref{language}. 

\begin{lemma}\label{language_finite}
Let $G$ be a virtually free group and let $M$ be a graded submonoid of $G$ generated by a finite set $S \subseteq G \setminus \{ 1 \}$. Then we can compute some $C > 0$ such that $WP_S(M) = \pi(L(\Gamma_S^C))$.
\end{lemma}

\begin{proof}
Since $M$ is f.g., we may assume that $G$ is f.g. Fix a finite generating set $A$ for $G$. We have already remarked that, given an explicit description of $G$ (say, a finite presentation), we can compute a polyhyperbolicity constant $K$ for $G = \langle A\rangle$. Since $G$ has decidable word problem, we can compute the constant $L$ from Theorem \ref{qg}. On the other hand, it follows from Lemma \ref{context-free} that $\Xi_S(u)$ is computable for every $u \in M$, and the same applies to the mapping $\zeta_{S,A}$. Hence the constant $L'$ from Therorem \ref{qg} is also computable, and so are $\lambda$ and $\varepsilon$. 

By \cite[Theorem H.1.7]{BH}, it is possible to compute a constant $R$ such that, for every $(\lambda,\varepsilon)$-quasi-geodesic $\varphi:[a,b] \to \overline{Cay}_A(G)$ and every geodesic $[\varphi(a),\varphi(b)]$ in $\overline{Cay}_A(G)$, the Hausdorff distance between their images is bounded by $R$ (that is, each point in one of them is at distance $\leq R$ from some point on the other). 

Let $(v,v') \in \pi(L(\Gamma_S)) \subseteq S^* \times S^*$. By Lemma \ref{language}, we have $\alpha_S(v) = \alpha_S(v')$. Write 
$v = v_1\ldots v_m$ and $v' = v'_1,\ldots, v'_n$ with $v_i,v'_j \in S$. Write $w_i = v_1\ldots v_i$ and $w'_i = v'_1\ldots v'_i$. By Theorem \ref{qg}, $\varphi_v$ and $\varphi_{v'}$ are both $(\lambda,\varepsilon)$-quasi-geodesics, hence their images lie at Hausdorff distance $\leq 2R$ from each other. Since each point in the image lies at distance $\leq \frac{L}{2}$ from a vertex, it follows that
\begin{equation}
\label{qua1}
\mbox{for each $i \in \{ 0,\ldots,m\}$, there exists some $j_i \in \{ 0,\ldots,n\}$ such that $d_A(w_i,w'_{j_i}) \leq R'$,}
\end{equation} 
where $R' = 2R + \frac{L}{2}$. Moreover, we may assume that $j_0 = 0$ and $j_m = n$ (since $\alpha_S(v) = \alpha_S(v')$, the quasi-geodesics $\varphi_v$ and $\varphi_{v'}$ are coterminal).

Consider now a maximal nondecreasing subsequence of $(j_0,\ldots,j_m)$, say $(j_{r_0},\ldots,j_{r_s})$. Suppose that $j_{r_i} - j_{r_{i-1}} > \zeta_{S,A}(2R'+L)$. Let
$$I = \{ h \in \{ r_{i-1} + 1,\ldots, r_i -1\} \mid j_{h} <  j_{r_{i-1}} \},\quad J = \{ h \in \{ r_{i-1} + 1,\ldots, r_i -1\} \mid j_{h} >  j_{r_{i}} \}.$$
By maximality of $(j_{r_0},\ldots,j_{r_s})$, we have $I \cup J = \{ r_{i-1} + 1,\ldots, r_i -1\}$.

Suppose that $r_i-1 \in I$. 
Let 
$x = (w'_{j_{r_i-1}})^{-1}w'_{j_{r_i}}$ and consider the sequence of vertices $w'_{j_{r_i-1}},w_{r_i-1},w_{r_i},w'_{j_{r_i}}$. 
Then 
$$d_A(1,x) = d_A(w'_{j_{r_i-1}},w'_{j_{r_i}}) \leq d_A(w'_{j_{r_i-1}},w_{r_i-1}) + d_A(w_{r_i-1},w_{r_i}) + d_A(w_{r_i},w'_{j_{r_i}}) \leq R' + L + R',$$
hence $\Xi(x) \leq \zeta_{S,A}(2R'+L)$. But $x = v'_{j_{r_i-1}+1}\ldots v'_{j_{r_i}} \in S^{j_{r_i} - j_{r_i-1}}$. Since $j_{r_i} - j_{r_i-1} > j_{r_i} - j_{r_{i-1}} > \zeta_{S,A}(2R'+L)$, we reach a contradiction. Thus $r_i-1 \in J$. Similarly, $r_{i-1}+1 \in I$.

Thus we may assume that there exists some $h \in I$ such that $h+1 \in J$. Taking 
$x = (w'_{j_{h}})^{-1}w'_{j_{h+1}}$  and the sequence of vertices $w'_{j_{h}}, w_{h}, w_{h+1}, w'_{j_{h+1}},$
we get a contradiction in a similar way. Therefore 
\begin{equation}
\label{qua2}
j_{r_i} - j_{r_{i-1}} \leq \zeta_{S,A}(2R'+L) \mbox{ for }i = 1,\ldots,m.
\end{equation}

Taking $y = w_{r_{i-1}}^{-1}w_{r_i}$, it follows from %=  (w_{r_{i-1}}^{-1}w'_{j_{r_{i-1}}}) (w'_{j_{r_{i-1}}})^{-1}w'_{j_{r_{i}}})((w'_{j_{r_{i}}})^{-1}w_{r_i},$$
(\ref{qua2}) that
$$\begin{array}{lll}
d_A(1,y)&=&d_A(w_{r_{i-1}},w_{r_i}) 
 \leq d_A(w_{r_{i-1}},w'_{j_{r_{i-1}}}) + d_A(w'_{j_{r_{i-1}}},w'_{j_{r_{i}}}) + d_A(w'_{j_{r_{i}}},w_{r_i})\\
 &\leq&R' + L\zeta_{S,A}(2R'+L) + R',
 \end{array}$$
hence $\Xi(y) \leq \zeta_{S,A}(L\zeta_{S,A}(2R'+L) +2R')$. Since $y = v_{r_{i-1}+1}\ldots v_{r_i} \in S^{r_i - r_{i-1}}$, we get
\begin{equation}
\label{qua3}
r_i - r_{i-1} \leq \zeta_{S,A}(L\zeta_{S,A}(2R'+L) +2R') \mbox{ for }i = 1,\ldots,m.
\end{equation}

Let 
$$C = R' + L\zeta_{S,A}(L\zeta_{S,A}(2R'+L) +2R') + L\zeta_{S,A}(2R'+L).$$

For $i = 0,\ldots,m$, write $t_i = w_{r_i}^{-1}w'_{j_{r_i}}$. It follows that $d_A(1,t_i) = d_A(w_{r_i},w'_{j_{r_i}}) \leq R' \leq C$. Moreover,
$$t_i = v_{r_i}^{-1}\ldots v_{r_{i-1}+1}^{-1}t_{i-1}v'_{j_{r_{i-1}}+1}\ldots v'_{j_{r_i}}.$$
Write 
$$z_{k,\ell} =  v_{k}^{-1}\ldots v_{r_{i-1}+1}^{-1}t_{i-1}v'_{j_{r_{i-1}}+1}\ldots v'_{\ell} \mbox{ for $k = r_{i-1}, \ldots, r_i$ and $\ell = j_{r_{i-1}},\ldots,j_{r_i}$}.$$
We have a path
\begin{equation}
\label{qua4}
t_{i-1} = z_{r_{i-1},j_{r_{i-1}}} \vvvvlongmapright{(v_{r_{i-1}+1}^{-1},1)} \ldots \vvlongmapright{(v_{r_{i}}^{-1},1)} z_{r_{i},j_{r_{i-1}}} \vvvvlongmapright{(1, v'_{j_{r_{i-1}}+1})} \ldots 
\vvlongmapright{(1, v'_{j_{r_{i}}})} z_{r_{i},j_{r_{i}}} = t_i
\end{equation}
in $\Gamma_S$. But 
$$d_A(1,z_{k,\ell}) \leq L(r_i-r_{i-1}) + R' + L(j_{r_i}-j_{r_{i-1}}) \leq C$$
for all $k,\ell$ in view of (\ref{qua2}) and (\ref{qua3}). Thus (\ref{qua4}) is a path in $\Gamma_S^C$. Thus, for $i = 1,\ldots,m$, we have a path $t_{i-1} \mapright{q_i} t_i$ in $\Gamma_S^C$ such that $\pi(q_i) = (v_{r_{i-1}+1},\ldots v_{r_{i}}, v'_{j_{r_{i-1}}+1} \ldots v'_{j_{r_{i}}})$. Gluing these paths in the obvious way, we get a successful path $1 = t_0 \mapright{q} t_m = 1$ in $\Gamma_S^C$ with $\pi(q) = (v_1\ldots v_m, v'_1\ldots v'_n) = (v,v')$.

Therefore $\pi(L(\Gamma_S)) \subseteq \pi(L(\Gamma_S^C))$. Since the opposite inclusion holds trivially, we only have to use Lemma \ref{language}.
\end{proof}

Thus for a graded submonoid $M$ of a virtually free group $G$ we can compute a finite automaton $\Gamma_S^C$ which encodes all the relations in $M$. We call it a {\it relation automaton} for $M$. Note that we can always suppose $\Gamma_S^C$ is trim if necessary.

\begin{theorem}\label{graded_rational}
Every graded submonoid of a virtually free group has rational word problem.
\end{theorem}

\begin{proof}
By Lemma \ref{language_finite}.
\end{proof}

Note that the converse of Theorem \ref{graded_rational} does not hold. If $M$ is a finite nontrivial monoid, then it has rational word problem by \cite[Theorem 4.2]{NPR}. However, $M$ is not graded since an epimorphism $\varphi:S^* \to M$ must necessarily have an infinite pre-image $\varphi^{-1}(x)$ for some $x \in M$.

We note also that Theorem \ref{graded_rational} does not hold for arbitrary graded monoids. Indeed, it is easy to check that the monoid $\mathbb{N} \times \mathbb{N}$ is graded (see the proof of Proposition \ref{raag}). However, in view of Lemma 4.5 and Theorem 6.2 from \cite{NPR}, $\mathbb{N} \times \mathbb{N}$ has not a rational word problem.

\section{The isomorphism problem}

Before solving the isomorphism problem for graded submonoids of virtually free groups, we need one additional lemma about rational subsets of a direct product $G \times G$.
Note that, even if $G$ is a free group, arbitrary rational subsets of $G \times G$, in fact even finitely generated subgroups, can behave very badly from the algorithmic point of view, in particular, they can have unsolvable membership problem, see \cite{Mikh}. However, one can decide whether a rational subset lies in the diagonal subgroup for a large class of groups, as the following lemma shows.
\par

\begin{lemma}\label{rational_subsets}
	Let $G$ be a f.g. group with decidable word problem, $R$ be a rational subset of $G \times G$, and $D=\{(x,x), \: x \in G \}$ be the diagonal subgroup of $G \times G$. Then it is decidable whether or not $R \subseteq D$.
\end{lemma} 

\begin{proof}

Fix a finite alphabet $A$ and a surjective homomorphism $\alpha:A^* \to G$. Then $\alpha$ induces a surjective homomorphism $\beta:\hat{A}^* \to G \times G$ defined by $\beta(a,b) = (\alpha(a),\alpha(b))$ for $(a,b) \in \hat{A}=(A \times \{ 1 \}) \cup (\{ 1\} \times A)$. We also write $\beta(a,b) = (\beta_1(a,b),\beta_2(a,b))$. 
 
Consider a finite automaton  $\mathcal{A} = (Q,\hat{A},I,T,E)$ such that $\beta(L(\mathcal{A})) = R$. We may assume that $\mathcal{A}$ is trim and has a unique initial state $q_0$.

Decidability may be proved using the construction of a product of an automaton by an action, as introduced in \cite[Definition 3.6, p.267]{Sak}. Namely, considering an action $\omega$ of $G \times G$ on $G$ given by $\omega(f,(g,h))=g^{-1}fh$ for all $f,g,h \in G$, one can see that $R$ is contained in $D$ if and only if the product $\mathcal{A} \times \omega$ is in bijection with $\mathcal{A}$, and the value of every terminal state in $\mathcal{A} \times \omega$ is $1$, where the value is defined on p.268 in \cite{Sak}, and these conditions are decidable.
	\par 

	But it is simpler for the reader to follow a direct proof, which avoids technical details and follows in fact the same ideas as the proof outlined above. Let $P$ be a spanning tree for $\mathcal{A}$, i.e., a subautomaton admitting, for every $q \in Q$, a unique path from $q_0$ to $q$. Note that such a spanning tree may be constructed fixing a total order on $A$ and taking, for every $q \in Q$, the path $q_0 \mapright{u_q} q$ such that $u_q$ is minimum for the shortlex order (and making $P$ the union of the edges occurring in these paths). 
	
	We define a mapping $\lambda:Q \to G$ as follows. Let $q \in Q$ and let $q_0 \mapright{u_q} q$ be the unique path in $P$ connecting $q_0$ to $q$. Then let $\lambda(q) = (\beta_1(u_q))^{-1}\beta_2(u_q)$.  
	We claim that
	\begin{equation}
	\label{spa1}
	\mbox{$R \subseteq D$ if and only if } \left\{
	\begin{array}{l}
	\lambda(t) = 1 \mbox{ for every } t \in T \\
	\lambda(q) = (\beta_1(x))^{-1}\lambda(p)\beta_2(x)\mbox{ for every }(p,x,q) \in E
	\end{array}
	\right.
	\end{equation}
	
	Indeed, suppose that $\lambda(t) \neq 1$ for some $t \in T$. Then $\beta_1(u_t) \neq \beta_2(u_t)$ and so $\beta(u_t) \notin D$. Since $u_t \in L(\mathcal{A})$, it follows that $R \not\subseteq D$.
	
	Suppose now that there exists some $(p,x,q) \in E$ such that $\lambda(q) \neq (\beta_1(x))^{-1}\lambda(p)\beta_2(x)$. Since $\mathcal{A}$ is trim, there exists some path $q \mapright{v} t$ for some $t \in T$. Hence $u_pxv, u_qv \in L(\mathcal{A})$, and to show that $R \not\subseteq D$ it suffices to show that $(\beta_1(u_pxv))^{-1} \beta_2(u_pxv) \neq (\beta_1(u_qv))^{-1} \beta_2(u_qv)$. Now
	$$(\beta_1(u_pxv))^{-1} \beta_2(u_pxv) = (\beta_1(v))^{-1}(\beta_1(x))^{-1}\lambda(p)\beta_2(x)\beta_2(v),$$
	$$(\beta_1(u_qv))^{-1} \beta_2(u_qv) = (\beta_1(v))^{-1}\lambda(q)\beta_2(v),$$
	so our goal follows from the inequality $\lambda(q) \neq (\beta_1(x))^{-1}\lambda(p)\beta_2(x)$.
	
	Conversely, assume that both conditions in the right hand side of (\ref{spa1}) hold. We show that
\begin{equation}
	\label{spa2}
	\mbox{if $q_0 \mapright{v} q$ is a path in $\mathcal{A}$, then $(\beta_1(v))^{-1}\beta_2(v) = \lambda(q)$.}
	\end{equation}

We use induction on $|v|$. The case $|v| = 0$ holds trivially, so we assume that $v = wx$, with $x \in \hat{A}$, and the claim holds for $w$. Then
we can split the path $q_0 \mapright{v} q$ into $q_0 \mapright{w} p \mapright{x} q$. Using the induction hypothesis and the right hand side of (\ref{spa1}), we get
$$(\beta_1(v))^{-1}\beta_2(v) = (\beta_1(x))^{-1}(\beta_1(w))^{-1}\beta_2(w)\beta_2(x) = (\beta_1(x))^{-1}\lambda(p)\beta_2(x) = \lambda_q,$$
so (\ref{spa2}) holds.

In particular, if $v \in L(\mathcal{A})$, we have a path  $q_0 \mapright{v} t$ for some $t \in T$, hence $(\beta_1(v))^{-1}\beta_2(v) = \lambda(t) = 1$ and so $\beta(v) \in D$. Therefore $R = \beta(L(\mathcal{A})) \subseteq D$ and so (\ref{spa1}) holds.
	
	Since $G$ has decidable word problem, both conditions in the right hand side of (\ref{spa1}) are decidable.
\end{proof}

We can now solve the {\it homomorphism problem}, which allows us to understand how easy it is to define homomorphisms by setting the images of a generating set.

\begin{theorem}[Homomorphism problem]
\label{homo}
Let $M$ be a graded submonoid of a virtually free group and let $G$ be a group with decidable word problem. Given a finite generating set $S$ of $M$ and a map $\varphi: S \rightarrow G$, one can decide if $\varphi$ can be extended to a homomorphism $\Phi: M \rightarrow G$. 
\end{theorem}

\begin{proof}
The mapping $\varphi$ induces a homomorphism $\overline{\varphi}: S^* \rightarrow G$ (where $S^*$ denotes the free monoid on $S$), as a well as a homomorphism $\widetilde{\varphi}: S^* \times S^* \rightarrow G \times G$ via $\wt{\varphi}((x,y))=(\overline{\varphi}(x), \overline{\varphi}(y))$ for all $x,y \in S^*$. Let $\alpha_S: S^* \to M$ and $\pi: \hat{S}^* \to S^* \times S^*$ be the canonical epimorphisms. 

By Lemma \ref{language_finite}, we can construct a finite automaton $\Gamma^C_S$ such that $WP_S(M) = \pi(L(\Gamma^C_S))$, making $WP_S(M)$ a rational subset of $S^* \times S^*$. Then $\wt{\varphi}(WP_S(M))$ is a rational subset of $G \times G$, since it is the image of a rational subset under a homomorphism. 

Let $D$ be the diagonal subgroup of $G \times G$. 
We claim that $\wt{\varphi}(WP_S(M)) \subseteq D$ if and only if there exists a homomorphism $\Phi: M \rightarrow G$ such that $\Phi\alpha_S=\overline{\varphi}$. This will solve the homomorphism problem. \par 
Indeed,  such a homomorphism exists if and only if for every $x,y \in S^*$ such that $\alpha_S(x)=\alpha_S(y)$ we have $\overline{\varphi}(x)=\overline{\varphi}(y)$: it is clear that it is a necessary condition, but it is also sufficient since if it's true then for every $z \in M$ one can define $\psi(z)=\overline{\varphi}(w)$ for any $w \in \alpha_S^{-1}(z)$, and this will give the desired homomorphism. This happens if and only if $\wt{\varphi}(WP_S(M)) \subseteq D$, decidable by Lemma \ref{rational_subsets}.
\end{proof}

We can use Theorem \ref{homo} to solve the isomorphism problem for graded submonoids of virtually free groups.

\begin{definition} 
{\it The isomorphism problem} for a class $\mathfrak{C}$ of f.g. monoids asks whether, given two monoids $M_1$ and $M_2$ in $\mathfrak{C}$ (given by their finite generating sets), one can decide whether or not $M_1$ is isomorphic to $M_2$. \par 
\end{definition}

\begin{theorem}\label{iso}
	The isomorphism problem is decidable for the class of graded submonoids of virtually free groups.
\end{theorem}

\begin{proof}
Let $M, N$ be two graded submonoids of virtually free groups $G_1$ and $G_2$. Since $M,N$ are f.g., we can assume that $G_1,G_2$ are finitely generated. Since it is well known that the class of (f.g.) virtually free groups is closed under free product, we may assume that $M,N$ are submonoids of the same f.g. virtually free group $G$. 

By Lemma \ref{graded_properties}, since $M$ and $N$ are graded, they are generated by their sets of irreducible elements which can be computed from the original generating sets as noticed above. Thus we can suppose $M$ is given by its set of irreducible elements $X= \{ x_1, \ldots, x_k \}$ and $N$ is given by its set of irreducible elements $Y= \{ y_1, \ldots, y_l \}$.
\par 
Note that the isomorphism between $M$ and $N$, if it exists, should induce a bijection between the sets of irreducible elements $X$ and $Y$. Thus, if $k \neq l$, we conclude that $M$ is not isomorphic to $N$.
\par 
Suppose now $k=l$. Let $\gamma_1, \ldots, \gamma_s$ be all bijections from $X$ to $Y$, then $\gamma_1^{-1}, \ldots, \gamma_s^{-1}$ are all bijections from $Y$ to $X$. For every $i=1, \ldots, s$ check whether the map $\gamma_i$ can be extended to a homomorphism $\overline{\gamma_i}$ from $M$ to $N$, or, equivalently, from $M$ to $G$; this can be checked by Theorem \ref{homo}. In the same way check whether the map $\gamma_i^{-1}$ can be extended to a homomorphism $\overline{\gamma_i^{-1}}$ from $N$ to $M$, or, equivalently, from $N$ to $G$. If for some $i=1, \ldots, s$, both extensions $\overline{\gamma_i}$ and $\overline{\gamma_i^{-1}}$ exist, then  $\overline{\gamma_i}$ is a bijection, and so it is an isomorphism from $M$ to $N$, with its inverse given by $\overline{\gamma_i^{-1}}$, and $M$ and $N$ are isomorphic. Otherwise $M$ and $N$ are not isomorphic.
This proves Theorem \ref{iso}.

\end{proof}

\par 

\section{Language-theoretic characterizations}

%A {\em generating system} of a monoid $M$ is a pair of the form $(X,\alpha)$, where $X$ is a set and $\alpha:X^* \to M$ is a morphism from the free monoid on $X$ onto $M$. 

Let $M$ be a monoid and let $\alpha:X^* \to M$ be a homomorphism of the free monoid on the set $X$ onto $M$. A mapping $\beta:X^* \to X^*$ is a {\em description} of $M$ with respect to $\alpha$ if $\beta(X^*)$ is a cross-section of $\alpha$ and $\alpha\beta = \alpha$, i.e., each element of $M$ is represented (through $\alpha$) by a unique element of $\beta(X^*)$ and each word $u \in X^*$ represents the same element of $M$ as $\beta(u)$.

We can view $\beta$ as a subset of $X^* \times X^*$ as $\beta = \{ (u,\beta(u)) \mid u \in X^*\}$. Then we say that $\beta$ is rational if it is rational as a subset of $X^* \times X^*$. This implies in particular that $X$ must be finite (and therefore $M$ is finitely generated). Moreover, this property does not depend on the finite generating system chosen, see \cite{Sak2}. A monoid is called {\it rational} if it has a rational description. See \cite{Sak2} for more details on descriptions and rational monoids. Recall a monoid $M$ is Kleene if the sets of rational and recognizable subsets of $M$ coincide.

Note that in general all rational monoids are Kleene \cite[Theorem 4.1]{Sak2}, but not all Kleene monoids are rational \cite{PS}. Moreover, a commutative monoid is rational if and only if it is Kleene \cite{Rup}.

%We say that $X \subseteq M$ is a {\em recognizable} subset of $M$ if there exists a monoid homomorphism $\theta:M \to K$ with $K$ finite satisfying $X = \theta^{-1}\theta(X)$. Equivalently, $X$ is recognizable if the {\em syntactic congruence} $\sim_L$ has finite index (this is the congruence on $M$ defined by $u \sim_L v$ if
%$$\forall p,q \in M\, (puq \in X \Leftrightarrow pvq \in X)).$$

%Let ${\rm Rat}(M)$ (respectively ${\rm Rec}(M)$) denote the set of all rational (respectively recognizable) subsets of $M$. If $M$ is finitely generated, then ${\rm Rec}(M) \subseteq {\rm Rat}(M)$ \cite[Proposition III.2.4]{Ber}. We say that $M$ is a {\em Kleene monoid} if ${\rm Rec}(M) = {\rm Rat}(M)$. This terminology arises from Kleene's Theorem (stating that ${\rm Rec}(A^*) = {\rm Rat}(A^*)$ whenever $A$ is finite \cite[Theorem III.2.1]{Ber}).

\begin{theorem}\label{rational_kleene}
Consider the three conditions
\begin{enumerate}
\item $M$ is a graded monoid,
\item $M$ is a rational monoid,
\item $M$ is a Kleene monoid,
\end{enumerate}
for a submonoid $M$ of a virtually free group $G$. Then 1 $\Rightarrow$ 2 $\Rightarrow$ 3. Equivalence holds if and only if $G$ is a free group.
\end{theorem} 

\begin{proof}

1 $\Rightarrow$ 2. Fix a finite subset $S \subseteq M \setminus \{ 1 \}$ generating $M$ as a monoid. We fix an arbitrary total order $<$ on $S$ and we consider the lexicographic order (the dictionary order) on $S^*$. We define a description $\beta:S^* \to S^*$ with respect to the canonical homomorphism $\alpha:S^* \to M$ as follows: given $u \in S^*$, $\beta(u)$ is the minimum element of $\alpha^{-1}\alpha(u)$ for the lexicographic order. 
Note that the lexicographic order is not a well-order, but we only need its restriction to $\alpha^{-1}\alpha(u)$, which is finite since $M$ is graded. We must show that $\beta$ is a rational subset of $S^* \times S^*$.

We show next that
\begin{equation}
\label{rk1}
\beta(S^*) \in {\rm Rat}(S^*).
\end{equation}

Recall the finite subautomaton $\Gamma_S^C$ arising from Lemma \ref{language_finite}, as well as the equality $WP_S(M) = \pi(L(\Gamma_S^C))$ and the notation $\pi$, $\pi_1$, $\pi_2$. Write $L = L(\Gamma^C_S)$ and $B = \beta(S^*)$.
%From a formal view point, the alphabet of this automaton is $Y = (X \times \{ 1\}) \cup (\{ 1\} \times X)$. Then we consider the homomorphisms
%$$\pi:Y^* \to X^* \times X^*, \quad \pi_1:Y^* \to X^*, \quad \pi_2:Y^* \to X^*$$
%defined by $\pi(x,x') = (x,x')$, $\pi_1(x,x') = x$ and $\pi_2(x,x') = x'$ for every $(x,x') \in Y$.
%Recall that by Lemma \ref{language_finite} the following set is rational:
%\begin{equation}
%\label{rk3}
%\pi(L(\Delta_M))= \{ (u,v) \in X^* \times X^* \mid \alpha(u) = \alpha(v) \}.
%\end{equation}
%It is easy to see that the set $$S= \{ (u,v) \in X^* \times X^* \mid u < v \}$$ is also rational.
%Note that $$\beta=R \setminus (X^* \times \pi_2(R \cap S)).$$
%Since rational subsets are closed under morphisms and products, and for free groups they are also closed under boolean operations, see \cite{BS1}, we get that $\beta$ is rational, as desired.
%We shall write $L = L(\Delta_M)$ and $B = \beta(X^*)$.

Let $K$ be the number of vertices of $\Gamma^C_S$. Suppose that there exists a path labelled by $(x_1,1)\ldots(x_K,1)$ in $\Gamma^C_S$, for some $x_i \in S$. Then there exists some loop at some vertex $p$ with label $(x_i,1)\ldots (x_j,1)$. with $0 \leq i \leq j \leq K$. Since $\Gamma^C_S$ may be assumed to be trim, we have paths from the basepoint $1$ to $p$ and back. Assume that the image by $\pi$ of their labels is respectively $(u,u')$ and $(v,v')$. It follows that $(u(x_i\ldots x_j)^nv,u'v') \in \pi(L)$ for every $n \geq 0$. Hence $u(x_i\ldots x_j)^nv \in \alpha^{-1}\alpha(u'v')$ for every $n \geq 0$, contradicting $M$ being graded. Thus there is no path labelled by $(x_1,1)\ldots(x_K,1)$ in $\Gamma^C_S$, and the same is true for paths labelled by $(1,x_1)\ldots(1,x_K)$.

Consider the following rational languages over $\hat{S}^*$:
$$\begin{array}{lll}
Z&=&\{ (x,1)(1,x) \mid x \in S\} \cup \{ (1,x)(x,1) \mid x \in S\},\\
T_1&=&\{ t \in (S \times \{ 1\})^* \mid |t| < K\},\\
T_2&=&\{ t \in (\{ 1\} \times S)^* \mid |t| < K\},\\
R&=&\{(x,1)t_1(1,x')y \mid x,x' \in S; x < x'; t_1 \in T_1; y \in \hat{S}^*\}\\
&\cup&\{(1,x')t_2(x,1)y \mid x,x' \in S; x < x'; t_2 \in T_2; y \in \hat{S}^*\}.
\end{array}$$
We claim that
\begin{equation}
\label{rk4}
B = S^* \setminus \pi_2(L \cap Z^*R). 
\end{equation}

Let $u \in B$. Suppose that $u = \pi_2(v)$ for some $v \in L \cap Z^*R$. Out of symmetry, we may assume that $v = z(x,1)t_1(1,x')y$, where $z \in Z^*$, $x,x' \in S$, $x < x'$, $t_1 \in T_1$ and $y \in \hat{S}^*$. It follows from the definitions that there exist $q,r,y_1,y_2 \in S^*$ such that $\pi(v) = (qxry_1,qx'y_2)$. Since $v \in L$ and $WP_S(M) = \pi(L)$, we have $\alpha(qxry_1) = \alpha(qx'y_2)$. Since $x < x'$, we have $qxry_1 < qx'y_2$ in the alphabetic order. But $qx'y_2 = \pi_2(v) = u \in B$, which is supposed to be the minimum element in $\alpha^{-1}\alpha(qx'y_2)$. In view of this contradiction, we deduce that $u \in S^* \setminus \pi_2(L \cap Z^*R)$ and so $B \subseteq S^* \setminus \pi_2(L \cap Z^*R)$.

Conversely, let $u \in S^* \setminus B$. Then $\beta(u) < u$. Suppose that $\beta(u)$ is a prefix of $u$. Writing $u = \beta(u)w$, we deduce from $\alpha\beta(u) = \alpha (u)$ that $\alpha(w) = 1$, which implies $\alpha^{-1}(1)$ infinite (since $w \neq 1$ and free monoids are torsion-free), contradicting the assumption that $M$ is graded. Hence $\beta(u)$ is not a prefix of $u$, and so there exist 
$x,x' \in S$ and $q,y_1,y_2 \in S^*$ such that $x < x'$ and $(\beta(u),u) = (qxy_1,qx'y_2)$. Note that there exists $z \in Z^*$ such that $\pi(z)=(q,q)$. In particular, $z \in L$. Since $\alpha\beta(u) = \alpha (u)$, we have $\alpha(q)\alpha(xy_1)=\alpha(q)\alpha(x'y_2)$ in the group $G$, so $\alpha(xy_1)=\alpha(x'y_2)$. Now it follows from $WP_S(M) = \pi(L)$ that there exists $w' \in L$ such that $\pi(w')=(xy_1,x'y_2)$. Then it follows from the above remarks that either $w'=(x,1)t_1(1,x')y$ for some $t_1 \in T_1$, $y \in \hat{S}^*$, or $w'=(1,x')t_2(x,1)y$ for some $t_2 \in T_2$, $y \in \hat{S}^*$. In both cases it follows that $w' \in L \cap R$, and so $zw' \in L \cap Z^*R$. 
Moreover, $\pi(zw') = \pi(z) \pi(w')=(q,q)(xy_1,x'y_2)=(qxy_1,qx'y_2)=(\beta(u),u)$. This implies that $u \in \pi_2(L \cap Z^*R)$ and so $S^* \setminus B \subseteq \pi_2(L \cap Z^*R)$. Therefore (\ref{rk4}) holds. 

Now $L \cap Z^*R \subseteq \hat{S}^*$ being rational follows from the standard  closure properties of rational languages \cite[Proposition I.4.2]{Ber}, hence $\pi_2(L \cap Z^*R) \subseteq S^*$ is rational since rational subsets are preserved by monoid homomorphisms \cite[Proposition II.2.2]{Ber}. Thus $B$ is rational since rational languages are closed under complement. Therefore (\ref{rk1}) holds.

Next we show that
\begin{equation}
\label{rk2}
\{ (u,\beta(u)) \mid u \in S^* \} = \pi(L \cap \pi_2^{-1}(B)).
\end{equation}

Indeed, let $u \in S^*$. We have $\alpha(u) = \alpha\beta(u)$. Since $WP_S(M) = \pi(L)$, there exists some $w \in L$ such that $\pi(w) = (u,\beta(u))$. Moreover, $\pi_2(w) = \beta(u) \in B$, hence $w \in L \cap \pi_2^{-1}(B)$. This proves the direct inclusion of (\ref{rk2}). 

Conversely, let $w \in L \cap \pi_2^{-1}(B)$. Since $WP_S(M) = \pi(L)$, we get $\pi(w) = (u,v)$ for some $u,v \in S^*$ such that $\alpha(u) = \alpha(v)$. But $\pi_2(w) \in B$ yields $v \in B$ and so $\beta(v) = v$. Thus $\alpha(u) = \alpha(v)$ yields $\beta(u) = \beta(v) = v$ and so $\pi(w) = (u,\beta(u))$. Therefore (\ref{rk2}) holds.

Now $L \subseteq \hat{S}^*$ is a rational language. Since $B \subseteq S^*$ is rational by (\ref{rk1}) and the inverse image of a rational language by a free monoid homomorphism is still a rational language \cite[Proposition I.4.2]{Ber}, and using also closure under intersection, it follows that $L \cap \pi_2^{-1}(B) \in {\rm Rat}(\hat{S}^*)$. Since rational subsets are preserved by monoid homomorphisms, it follows that $\{ (u,\beta(u)) \mid  u \in S^* \}$ is a rational subset of $S^* \times S^*$ as required.

2 $\Rightarrow$ 3. By \cite[Theorem 4.1]{Sak2}.

We complete the proof by showing that
\begin{equation}
\label{newe2}
\mbox{ 3 $\Rightarrow$ 1 if and only if $G$ is a free group.}
\end{equation}

Assume first that $G$ is a free group and $M$ is a Kleene monoid.
Recall that $\sim_L$ denotes the syntactic congruence for $L \subseteq M$, in particular, for $g \in M$, $\sim_g$ is the syntactic congruence for the set $\{ g \}$. Suppose that $M$ contains an invertible element $u \neq 1$.
Clearly, the singleton set $\{ 1 \}$ is rational, hence recognizable by (iii). Let $k,m \geq 0$ be distinct. We have $1\cdot u^k \cdot (u^{-1})^k = 1$ but $1\cdot u^m \cdot (u^{-1})^k \neq 1$ (since $u \in G\setminus \{ 1\}$ must have infinite order). Thus $u^k \not\sim_1 u^m$ and so $\sim_1$ has infinite index, contradicting $1 \in {\rm Rec}(M)$. Therefore $1$ is the unique invertible element of $M$.

Now consider a finite set $X$ and a surjective monoid homomorphism $\varphi:X^* \to M$. We may assume that $1 \notin \varphi(X)$. Suppose that there exists some $u \in M$ such that $\varphi^{-1}(u)$ is infinite. On the other hand, $\{ u \} \in {\rm Rat}(M) = {\rm Rec}(M)$ by (iii), hence $\sim_{u}$ has finite index, say $k$. Since $X$ is finite and $\varphi^{-1}(u)$ is infinite, there exists some $v \in \varphi^{-1}(u)$ with length $> k$. For $i = 0,\ldots,k$, let $v_i$ denote the prefix of length $i$ of $v$, and write $v = v_iv'_i$.

Suppose that $\varphi(v_i) \sim_u \varphi(v_j)$ with $0 \leq i < j \leq k$. Since $1 \cdot \varphi(v_i) \cdot \varphi(v'_i) = u$, then 
$1 \cdot \varphi(v_j) \cdot \varphi(v'_i) = u$ and so $\varphi(v_j)\varphi(v'_i) = \varphi(v_i)\varphi(v'_i)$. Since $M$ (being a submonoid of $G$) is cancellative, it follows that $\varphi(v_i) = \varphi(v_j)$. Now $v_i$ is a proper prefix of $v_j$, say $v_j = v_iz$, hence $\varphi(z) = 1$. Let $x$ be the first letter of $z$ and write $z = xz'$. Then $\varphi(x)\varphi(z') = 1$. Since $\varphi(x) \neq 1$ by assumption, then $\varphi(x)$ would be an invertible element of $M$ different from the identity, contradicting our previous conclusion. Therefore $\varphi(v_i) \not\sim_u \varphi(v_j)$ whenever $0 \leq i < j \leq k$, contradicting $|M/\sim_u| = k$. Thus $\varphi^{-1}(u)$ is finite for every $u \in M$ and so $M$ is graded.

Assume now that $G$ is not a free group. By a theorem of Stallings \cite{Sta2}, a f.g. virtually free group is free if and only if it is torsion-free. Hence $G$ has some element $g \neq 1$ of finite order. Set $M = \langle g\rangle$. Since $M$ is finite, it is obviously a Kleene monoid. However $M$ is not graded by Lemma \ref{graded_properties}(2). Note that the implication 2 $\Rightarrow$ 1 also fails in this case since any finite nontrivial group is rational by \cite{Sak2}.
\end{proof}

For non virtually free groups, the implication $1 \Rightarrow 3$ (and therefore $1 \Rightarrow 2$) may fail in view of the following result.

Recall that a {\em right-angled Artin group} is a group of the form $F/N$, where $F$ is the free group on some finite set $A$ and $N$ is the normal subgroup generated by some subset of $\{ [a,b] \mid a,b \in A \}$.

\begin{proposition}

\label{raag}

The following conditions are equivalent for a right-angled Artin group $G$:

\begin{enumerate}

\item

Every graded submonoid of $G$ is a Kleene monoid.

\item

$G$ is a free group.

\end{enumerate}

\end{proposition}

\begin{proof}

$1 \Rightarrow 2$. Suppose that $G$ is not a free group. Then $\mathbb{Z} \times \mathbb{Z}$ is a submonoid of $G$, and so is $\mathbb{N} \times \mathbb{N}$.

Clearly, the identity $(0,0)$ is the unique idempotent and the unique invertible element of the f.g. (additive) monoid $\mathbb{N} \times \mathbb{N}$. If $(m,n) \in \mathbb{N} \times \mathbb{N}$, then $(m,n)$ has precisely $(m+1)(n+1)$ factors, hence $\mathbb{N} \times \mathbb{N}$ is finite $\cal{J}$-above and therefore graded by Proposition \ref{eqgrad}.

Consider the rational subset $L = (1,1)^* = \{ (n,n) \mid  n \in \mathbb{N} \}$ of $\mathbb{N} \times \mathbb{N}$. If $m,n \in \mathbb{N}$ are distinct, then $(0,m) \not\sim_L (0,n)$ because $(0,m)+(m,0) \in L$ but $(0,n)+(m,0) \notin L$. Thus $\sim_L$ has infinite index and so $L$ is not a recognizable subset of $\mathbb{N} \times \mathbb{N}$. Therefore $\mathbb{N} \times \mathbb{N}$ is not a Kleene monoid and $G$ fails condition 1.

$2 \Rightarrow 1$. By Theorem \ref{rational_kleene}.

\end{proof}

\section{Open problems}

%We don't know yet whether these problems are decidable for arbitrary f.g. submonoids of a (virtually) free group.
\begin{question}
	Is the homomorphism problem decidable for every f.g. submonoid of a (virtually) free group? 
\end{question}

\begin{question}
	Is the isomorphism problem decidable for the class of all f.g. submonoids of a (virtually) free group? 		
\end{question}

\begin{question}
	Given a f.g. submonoid of a free group, can one decide whether it embeds in a free monoid or not? Note that it should be graded if it does, but this is not sufficient, see Example \ref{ex1}.
\end{question}

 \section*{Acknowledgements}

We are grateful to the referees for insightful comments and suggestions which have improved the paper.

Both authors were partially supported by CMUP (UID/MAT/00144/2013), which is funded by FCT (Portugal) with national (MEC) and European structural funds (FEDER), under the partnership agreement PT2020. The second author also thanks the Russian Foundation for Basic Research (project no. 15-01-05823). %, by the ERC Grant 336983, by the Basque Government grant IT974-16 and by the grant MTM2014-53810-C2-2-P of the Ministerio de Economia y Competitividad of Spain.

\end{document}